\documentclass[11pt]{amsart}
\usepackage[utf8]{inputenc}
\usepackage[margin=2.7cm]{geometry}
\usepackage{graphicx,xcolor}
\usepackage{amscd,amsmath,amsfonts,amssymb,amsthm}

\newcommand{\N}{{\mathbb N}}
\newcommand{\R}{{\mathbb R}}
\newcommand{\W}{H_0^1(\Omega)}

\newtheorem{theorem}{Theorem}
\newtheorem{lemma}[theorem]{Lemma}
\newtheorem{proposition}[theorem]{Proposition}

\numberwithin{equation}{section}

\begin{document}

\title[]{Existence of solutions to four-dimensional Kirchhoff problems with critical-concave nonlinearities}

\author{Giovanni Anello}
\address[G. Anello]{Department of Mathematical and Computer Sciences, Physical Sciences and Earth Sciences\\
University of Messina\\
Viale F. Stagno d’Alcontres, 31 - 98166 Messina, Italy}
\email{\tt ganello@unime.it}

\author{Luca Vilasi$^\dagger$}
\address[L. Vilasi]{Department of Mathematical and Computer Sciences, Physical Sciences and Earth Sciences\\
University of Messina\\
Viale F. Stagno d’Alcontres, 31 - 98166 Messina, Italy}
\email{\tt lvilasi@unime.it}

\keywords{Kirchhoff problem; Brezis-Nirenberg problem; dimension four; existence; critical-concave nonlinearity\\
\phantom{aa} 2020 AMS Subject Classification: 35J20, 35J25, 35J61\\
\phantom{aa} $^\dagger$Corresponding author
}

\begin{abstract}
We deal with a Kirchhoff problem on a smooth bounded domain of $\R^4$ with competing critical and concave terms. By using new approximation techniques and the Nehari manifold analysis,
we derive several existence results, complementing earlier ones obtained in the critical-convex case. Compared to similar results in the literature, we provide explicit bounds of the range of the parameters leading to the existence of solutions.
\end{abstract}

\maketitle

%%%%%%%%%%%%%%%%%%%%%%%%%%%%%%%%%%%%%%%%
\section{Introduction}
In this paper we analyze the following nonlocal elliptic problem,
\begin{equation}\tag{$P_{\lambda,q}$}\label{problem}
        \left\{\begin{array}{ll}
            \displaystyle{-\left( a+b\left(\int_{\Omega}|\nabla u|^2\, dx\right)^{\frac{q-2}{2}}\right)\Delta u= u^{q-1}+\lambda u^{p-1}} & \text{ in } \Omega, \smallskip\\
            u> 0 & \text{ in } \Omega, \smallskip\\
            u=0 & \text{ on } \partial\Omega,
        \end{array}\right.
\end{equation}
where $\Omega\subset \mathbb{R}^4$ is a smooth bounded domain, $a,b,\lambda$ are positive parameters, $p\in(1,2)$, $q\in (2,4]$. This type of problems, in which the leading operator is coupled with an integral coefficient, historically dates back to the work of G. Kirchhoff \cite{kir1883vorlesungen}, who first proposed the equation
\begin{equation}\label{kirchhoff}
    \varrho \frac{\partial^2 u}{\partial t^2} - \left( \frac{P_0}{h}+\frac{E}{2L}\int_0^L \left| \frac{\partial u}{\partial x}\right| ^2 dx\right) \frac{\partial^2 u}{\partial x^2} = 0
\end{equation}
as a nonlinear extension of d'Alembert's wave equation for free vibrations of elastic strings (in \eqref{kirchhoff}, $\varrho$ is the mass density, $u=u(x,t)$ the transverse string displacement at the space coordinate $x$ and time $t$, $P_0,h,E,L$ constants whose physical meaning is illustrated in \cite{av}).

Among all the directions of study of such problems, that of the critical growth in dimension four is particularly intriguing, due to the existence of an interaction between the higher order Kirchhoff term and the critical nonlinearity  (when $N=4$, the Sobolev critical exponent $2^*=2N/(N-2)$ equals 4). This obviously makes their variational treatment more difficult and motivates the growing interest in this case. Some interesting contributions along this direction come from \cite{f,lizhe2025existence,nai2014the,nai2014positive,naishi2020existence}, where the existence of solutions has been obtained by different variational techniques. Other papers related to \eqref{problem}, where, in the spirit of \cite{ambbrecer1994combined}, the effect of competing (weighted) concave-convex terms in producing solutions has been studied, are \cite{benmatlit2019multiplicity,chekuowu2011the}. 

In the recent \cite{av}, problem \eqref{problem} has been considered in the case $q=4$ and $p\in(2,4)$. The aim of this paper is to carry out the analysis in the presence of a concave perturbation $u^{p-1}$, $p\in(1,2)$, by means of the same methods as \cite{av}, i.e.,  a specific approximation procedure in combination with the Nehari manifold scheme. The strength of our approach, besides providing existence results via novel techniques, relies on the explicit computation of the threshold of the parameters for which solutions are obtained, unlike earlier results addressing the same topic (cf. \cite{f,naishi2020existence,yaomu2016multiplicity}). This fact is of great relevance also in relation to the (delicate) linear perturbation case $p=2$, which we have in mind  to consider as a limit case of the results obtained here and treat in a subsequent paper.

%Among the contributions mentioned above, we recall \cite{naishi2020existence}, based on the Nehari manifold method and the concentration-compactness analysis for Palais–Smale sequences, and \cite{yaomu2016multiplicity} where the Krasnoselskii genus theory is used to derive the existence of infinitely many solutions in the critical-concave

Before stating our results, we need to introduce first some notations.

Let $\left\|\cdot\right\|_m$,  $m\in[1,4]$, be the canonical $L^m$-norm on $\Omega$, $\left\|\cdot \right\|:=\left\| \nabla \cdot\right\|_2$ and
\begin{equation}\label{costanteimm}
c_m:=\sup_{u\in\W,\, \|u\|=1}\left\|u\right\|_m, \quad m\in[1,4].
\end{equation}
We also set $S:=c_4^{-2}$ and throughout the paper we assume that $$
0<b<S^{-2}.
$$
A simple argument shows that $c_q^q>b$ for every $q$ in a left neighborhood of $4$ (cf. Proposition \ref{ineqq} below). For these $q$'s and for every $p\in(1,2)$, let us define the following constants,
\begin{equation}\label{lambdapq}
    \Lambda_{p,q}:=\sup_{t>0}\left[\frac{at^2}{\sup_{\|u\|=t}\left(\|u\|_p^p+\|u\|_q^q\right)-bt^q}\right]^{\frac{q-p}{q-2}},
\end{equation}
\begin{equation}
    \Lambda_p:=\frac{2 a^\frac{4-p}{2}}{c_p^p(4-p)(1-bS^2)}\left(\frac{pS^2}{2-p}\right)^{\frac{2-p}{2}},
\end{equation}
\begin{equation}\label{lambdahatpq}
    \widehat{\Lambda}_{p,q}:=\frac{p}{2}\cdot\frac{d_{p,q}}{\widetilde{c}_{p,q}^p(c_q^q-b)^\frac{2-p}{q-2}},
\end{equation}
where %$d_{p,q}$ is defined in \eqref{dpq} and
\begin{equation}\label{dpq}
    d_{p,q}:= \displaystyle\left(\frac{a}{q-p}\right)^\frac{q-p}{q-2}(q-2)\left(2-p\right)^\frac{2-p}{q-2},
\end{equation}
\begin{equation*}
    \widetilde{c}_{p,q}:=\sup_{u\in\mathcal{B}_q}\|u\|_p,
\end{equation*}
%$\mathcal{B}_q$ being defined in \eqref{bq}.
\begin{equation}\label{bq}
    \mathcal{B}_q:=\{u\in \W: \|u\|=1 \ \text{ and }\ \|u\|_q^q>b\}
\end{equation}
(the consistency of these definitions is justified in the following section). We denote by
\begin{equation}
    I_{\lambda,q}(u):=\frac{a}{2}\left\| u\right\|^2 + \frac{b}{q}\left\| u\right\|^q - \frac{1}{q}\left\| u\right\|^q_q -\frac{\lambda}{p}\left\| u\right\|^p_p, \quad u\in\W,
\end{equation}
the $C^1$--energy functional corresponding to \eqref{problem}. By a weak solution to \eqref{problem} we mean any $u\in\W$, positive in $\Omega$, satisfying
$$
\left(a+b\left( \int_\Omega |\nabla u|^2 dx\right)^\frac{q-2}{2} %\left\| u\right\|^{q-2}
\right)\int_\Omega \nabla u\nabla v dx -\int_\Omega u^{q-1} v dx -\lambda\int_\Omega u^{p-1} vdx =0,
$$
for every $v\in\W$. Finally, letting $J_{\lambda,q}:\W\rightarrow \R$ be the functional
\begin{equation}
    J_{\lambda,q}(u):=I_{\lambda,q}'(u)(u)=a\left\| u\right\|^2 +b\left\| u\right\|^q - \left\| u\right\|^q_q -\lambda \left\| u\right\|^p_p, \quad u\in\W,
\end{equation}
we define the Nehari manifold $\mathcal{N}_{\lambda,q}$ corresponding to $I_{\lambda,q}$ and its subsets $\mathcal{N}_{\lambda,q}^+$, $\mathcal{N}_{\lambda,q}^-$ by
\begin{equation}
    \begin{split}
        \mathcal{N}_{\lambda,q} &:=\{u\in \W\setminus\{0\}: J_{\lambda,q}(u)=0\},\smallskip\\
        \mathcal{N}_{\lambda,q}^+ &:=\{u\in \mathcal{N}_{\lambda,q}:  J'_{\lambda,q}(u)(u)>0\},\smallskip\\
        \mathcal{N}_{\lambda,q}^- &:=\{u\in \mathcal{N}_{\lambda,q}: J'_{\lambda,q}(u)(u)<0\}.
    \end{split}
\end{equation}

\medskip

Our main results read as follows.

\begin{theorem}\label{subcrit}
    Assume $q\in(2,4)$ and $\lambda\in(0,\Lambda_{p,q})$. Then, problem \eqref{problem} has a solution $u_{\lambda,q}\in \mathcal{N}_{\lambda,q}^+$ such that
    \begin{equation}\label{infI}
        I_{\lambda,q}(u_{\lambda,q})=\inf_{\mathcal{N}_{\lambda,q}^+}I_{\lambda,q}.%=\tilde{I}_{\lambda,q}(u_{\lambda,q})=\inf_{\mathcal{A}_{\lambda,q}^*}\tilde{I}_{\lambda,q}<0.
    \end{equation}
\end{theorem}

\begin{theorem}\label{crit}
Assume $\lambda\in\left(0,\min\left\lbrace \Lambda_{p,4}, \Lambda_p\right\rbrace\right)$. Then, problem $(P_{\lambda,4})$ has a solution $u_{\lambda,4}\in \mathcal{N}_{\lambda,4}^+$ such that
    \begin{equation*}
        I_{\lambda,4}(u_{\lambda,4})=\inf_{\mathcal{N}_{\lambda,4}^+}I_{\lambda,4}<0.
    \end{equation*}
\end{theorem}

\begin{theorem}\label{solutionN-subcrit}
    Assume $q\in(2,4)$ and $\lambda\in\left(0,\min\{\Lambda_{p,q},\widehat{\Lambda}_{p,q}\}\right)$. Then, there exists $q_\lambda\in(2,4)$ such that, for every $q\in [q_\lambda,4)$,
    %, let $q_\lambda$ be as in Lemma \ref{stimasuperioreinf}, and assume $q\in [q_\lambda,4)$.
    problem \eqref{problem} has a solution $u_{\lambda,q}\in\mathcal{N}_{\lambda,q}^-$ such that
    $$
    I_{\lambda,q}(u_{\lambda,q})=\inf_{\mathcal{N}_{\lambda,q}^-}I_{\lambda,q}>0.
    $$
\end{theorem}

\begin{theorem}\label{soluzcriticN-}
    There exists $\sigma>0$ such that, if $b\in \left(S^{-2}-\sigma,S^{-2}\right)$ and $\lambda\in\left(0,\min\{\Lambda_{p,4},\widehat{\Lambda}_{p,4} \}\right)$, problem $(P_{\lambda,4})$ has a
    solution $v_{\lambda,4}\in\mathcal{N}_{\lambda,4}^-$, such that
    $$
    I_{\lambda,4}(v_{\lambda,4})=\inf_{\mathcal{N}_{\lambda,4}^-}I_{\lambda,4}>0.
    $$
\end{theorem}

We point out that the number $\sigma$ in Theorem \ref{soluzcriticN-}, which determines the size of the range of $b$, is not only deduced, but can be computed as a function of other quantities appearing in auxiliary lemmas (see \eqref{sigma} of Section \ref{minNmeno}).

As anticipated before, the proof of Theorems \ref{subcrit} and \ref{solutionN-subcrit}, dealing with the subcritical regime,  relies upon the approximation procedure employed in \cite{av}. More specifically, the underlying idea is to introduce suitable functionals $\widetilde{I}_{\lambda,q}$, $K_{\lambda,q}$, approximating $I_{\lambda,q}$ and its second derivative and depending only on $L^p$ and $L^q$-norms, as well as related subsets of the Nehari manifold. Variational arguments based upon direct minimization allow us to find critical points of $\widetilde{I}_{\lambda,q}$, and hence of $I_{\lambda,q}$, since the former agrees with the latter on $\mathcal{N}_{\lambda,q}$.
We emphasize that, while working on $\mathcal{N}^+_{\lambda,q}$ we benefit from the coercivity of $\widetilde{I}_{\lambda,q}$ %on a set containing $\mathcal{N}^+_{\lambda,q}$
(see the proof of Theorem \ref{subcrit}), this fails for the minimization on $\mathcal{N}^-_{\lambda,q}$, making this case more delicate. We circumvent the obstacle by introducing an $L^q$-constraint, which through careful estimates of the parameters involved is finally "reabsorbed", allowing us to conclude as in the previous case. A main role in this case is played by the fibering map associated with the truncated instanton \eqref{truncinstanton}.

The solutions which stem from Theorems \ref{crit} and \ref{soluzcriticN-}  are obtained instead as a limit of a sequence of solutions of the corresponding subcritical problems, making also use of Lions' concentration-compactness principle. We stress out that, while the solution on $\mathcal{N}_{\lambda,4}^+$ is obtained for $b$ ranging all over $(0,S^{-2})$, the minimization on $\mathcal{N}_{\lambda,4}^-$ requires $b$ close enough to $S^{-2}$. In this sense, our result complements Theorem 1.4 of \cite{naishi2020existence}, in which the existence of a solution for $(P_{\lambda,4})$ was established for $b$ tending to 0 and for every $\lambda>0$, and for any $b\in(0,S^{-2})$ and sufficiently small $\lambda=\lambda(b)$ (but with no information on the parameter bounds). %Compared to this work, based on the Nehari manifold method and the concentration-compactness analysis for Palais–Smale sequences, our results, as already said we make explicit the interval of the parameters leading to the existence results.
Similarly to \cite{naishi2020existence}, instead, we are not able to say if the range of $b$ in Theorem \ref{soluzcriticN-}  can be extended to the whole $(0,S^{-2})$, it remains an open problem.

The following sections are structured as follows. Section \ref{preliminary} contains the preparatory results and the analysis of the Nehari manifold useful for our purposes. In Section \ref{minNpiu} we address the minimization on $\mathcal{N}^+_{\lambda,q}$, both in the subcritical and in the critical case; in Section \ref{minNmeno} we do the same thing on
$\mathcal{N}^-_{\lambda,q}$.

%%%%%%%%%%%%%%%%%%%%%%%%%%%%%%%%%%%%%
\section{Preliminary results and Nehari manifold analysis}\label{preliminary}

\begin{proposition}\label{ineqq}
There exist $q_0\in (2,4)$ such that $c_q^q>b$ for each $q\in [q_0,4]$.
\end{proposition}

\begin{proof}
By Theorem 9 of \cite{afi}, we know that  the function $q\in (2,4]\mapsto \sup_{\|u\|=1}\|u\|^q_q$ is continuous.  Since
$S^{-2}=\sup_{\|u\|=1}\|u\|_4^4>b$, the conclusion easily follows.
\end{proof}

From now on, even if not explicitly stated, we always assume $q\in [q_0,4]$.
\begin{proposition}\label{ineq}
For each $t>0$, one has
$$
\sup_{\|u\|=t}\left(\|u\|_p^p+\|u\|_q^q\right)-b t^q\geq t^q(c_q^q-b)>0.
$$
\end{proposition}

\begin{proof}
It is a direct consequence of Proposition \ref{ineqq}.
\end{proof}

%Let us introduce the non-empty set
%\begin{equation}\label{bq}
%\mathcal{B}_q:=\{u\in \W: \|u\|=1 \ \text{ and }\ \|u\|_q^q>b\}.
%\end{equation}

%By Proposition \ref{ineq}, it is well-defined the following number:
%\begin{equation}\label{lambdapq}
%\Lambda_{p,q}:=\sup_{t>0}\left[\frac{at^2}{\sup_{\|u\|=t}\left(\|u\|_p^p+\|u\|_q^q\right)-bt^q}\right]^{\frac{q-p}{q-2}}.
%\end{equation}

\begin{lemma}\label{Lambdabound}
If $p\in(1,2]$ one has
\begin{equation}\label{ineq2}
\frac{d_{p,q}}{ c_p^p \left( c_q^q -b \right)^\frac{2-p}{q-2}} < \Lambda_{p,q} \leq \frac{d_{p,q}}{\sup_{u\in\mathcal{B}_q} \left\| u\right\|_p^p
\left( \left\|u\right\|_q^q -b \right)^\frac{2-p}{q-2}}\leq \Lambda_*,
\end{equation}
where $c_m$, $\Lambda_{p,q}$, $d_{p,q}$, $\mathcal{B}_q$ are as in  \eqref{costanteimm}, \eqref{lambdapq}, \eqref{dpq} and \eqref{bq}, respectively,
%\begin{equation}\label{dpq}
%d_{p,q}:=
%\left\{
%\begin{array}{ll}
%\displaystyle\left(\frac{a}{q-p}\right)^\frac{q-p}{q-2}(q-2)\left(2-p\right)^\frac{2-p}{q-2}, & \text{ if }\ 1<p<2,\smallskip\\
%a & \text{ if } \ p=2,
%\end{array}\right.
%\end{equation}
and
\begin{equation}\label{L*}
\Lambda_*:=\sup_{(m,r)\in [1,2]\times [q_0,4])}\frac{d_{m,r}}{\sup_{u\in\mathcal{B}_r} \left\| u\right\|_m^m \left( \left\|u\right\|_r^r -b
\right)^\frac{2-m}{r-2}} < +\infty.
\end{equation}
Moreover, one has
\begin{eqnarray}\label{ineq3}
\Lambda_{p,4}\leq \liminf_{q\rightarrow 4^-} \Lambda_{p,q}.
\end{eqnarray}
\end{lemma}

\begin{proof}
If $p\in(1,2)$ and $u\in \mathcal{B}_q$, the function
$$
\eta(t):= t^{p-2}\|u\|_p^p+t^{q-2}(\|u\|_q^q-b), \quad t\in(0,+\infty),
$$
is decreasing in $(0,t_*)$ and increasing in $(t_*,+\infty)$, where
$$
t_*:=\left( \frac{(2-p)\left\| u\right\|_p^p}{(q-2)\left( \left\|u\right\|_q^q -b \right)}\right)^\frac{1}{q-p},
$$
and
$$
\eta(t_*)=\left( \frac{q-p}{q-2}\right)\left(\frac{2-p}{q-2}\right)^\frac{p-2}{q-p}\left\| u\right\|_p^\frac{p(q-2)}{q-p} \left( \left\|u\right\|_q^q
-b \right)^\frac{2-p}{q-p},
$$
while, for $p=2$, $\eta$ is increasing in $(0,+\infty)$. Therefore, one has
\begin{align*}
\Lambda_{p,q} &=\sup_{t>0}\left[\frac{a}{\sup_{\|u\|=1}\left(t^{p-2}\|u\|_p^p+t^{q-2}(\|u\|_q^q-b)\right)}\right]^{\frac{q-p}{q-2}}\\
& \leq \sup_{t>0}\inf_{u\in\mathcal{B}_q}\left[\frac{a}{\left(t^{p-2}\|u\|_p^p+t^{q-2}(\|u\|_q^q-b)\right)}\right]^{\frac{q-p}{q-2}}\\
& \leq \inf_{u\in\mathcal{B}_q}\sup_{t>0}\left[\frac{a}{\left(t^{p-2}\|u\|_p^p+t^{q-2}(\|u\|_q^q-b)\right)}\right]^{\frac{q-p}{q-2}}\\
& = \frac{d_{p,q}}{\displaystyle{\sup_{u\in\mathcal{B}_q}}\left( \left\| u\right\|_p^p \left( \left\|u\right\|_q^q -b
\right)^\frac{2-p}{q-2}\right)},
\end{align*}
and thus, the second inequality in \eqref{ineq2}. The third inequality and the fact that $\Lambda_*<+\infty$ are trivial.  Now, let us prove the first inequality in \eqref{ineq2}. To this end, we first claim that
\begin{equation}\label{sineq}
\sup_{\|u\|=t}(\|u\|_p^p+\|u\|_q^q)<\sup_{\|u\|=t}\|u\|_p^p+\sup_{\|u\|=t}\|u\|_q^q.
\end{equation}
Indeed, assuming
$$
\sup_{\|u\|=t}(\|u\|_p^p+\|u\|_q^q)=\sup_{\|u\|=t}\|u\|_p^p+\sup_{\|u\|=t}\|u\|_q^q,
$$
one would have
$$
t^pc_p^p+t^qc_q^q=\sup_{\|u\|=t}\|u\|_p^p+\sup_{\|u\|=t}\|u\|_q^q=\sup_{\|u\|=t}(\|u\|_p^p+\|u\|_q^q).
$$
Hence, we can fix a sequence $\{u_n\}$ in $\W$ such that $\|u_n\|=1$ and
\begin{equation*}
t^p\|u_n\|_p^p+t^q\|u_n\|_q^q\geq t^pc_p^p+t^qc_q^q-\frac{1}{n},
\end{equation*}
which implies
\begin{equation}\label{1}
\|u_n\|_p^p\geq c_p^p -\frac{1}{nt^p} \quad \text{and} \quad \|u_n\|_q^q\geq c_q^q -\frac{1}{nt^q},
\end{equation}
for each $n\in \N$. Without loss of generality, we may assume $u_n$ non-negative. Since $\|u_n\|=1$ for each $n\in \N$, there exists some non-negative function $u_*\in \W$ such that $u_n\rightarrow u_*$ weakly in $\W$, and strongly in $L^m(\Omega)$, for each $m\in[1,4)$. In particular, from
\eqref{1} one has
$$
\|u_*\|_p^p =\lim_{n\rightarrow +\infty}\|u_n\|_p^p=c_p^p,
$$
and thus (it is well known that) it must be $\|u_*\|=1$. This implies that $\{u_n\}$ strongly converges to $u_*$  in $\W$.  Hence, again from \eqref{1}, it follows (also for $q=4$) that
$$
\|u_*\|_q^q =\lim_{n\rightarrow +\infty}\|u_n\|_q^q=c_q^q.
$$
Consequently, by applying the Lagrange Multipliers Theorem, we infer that $u_*$ satisfies
$$
\int_{\Omega}(\nabla u_*\nabla \varphi - c_p^{-p}u_*^{p-1}\varphi)dx=\int_{\Omega}(\nabla u_*\nabla \varphi - c_q^{-q}u_*^{q-1}\varphi)dx=0,
$$
and in particular
$$
\int_{\Omega}(c_q^{-q}u_*^{q-1}- c_p^{-p}u_*^{p-1})\varphi dx=0,
$$
for each $\varphi\in \W$. This implies $u_*(x)=\left(\frac{c_q^q}{c_p^p}\right)^{\frac{1}{q-p}}$ in $\Omega$, which is a contradiction.

Thus, the strict inequality \eqref{sineq} holds. Now, arguing as before we obtain
\begin{align*}
\Lambda_{p,q} & >\sup_{t>0}\left[\frac{at^2}{\sup_{\|u\|=t}\|u\|_p^p+\sup_{\|u\|=t}\|u\|_q^q-bt^q}\right]^{\frac{q-p}{q-2}}\\
&=\sup_{t>0}\left[\frac{a}{c_p^p t^{p-2}+(c_q^q-b)t^{q-2}}\right]^{\frac{q-p}{q-2}}\\
& =\frac{d_{p,q}}{ c_p^p \left( c_q^q -b \right)^\frac{2-p}{q-2}},
%=(ac_p^p)^{\frac{q-p}{2}}.
\end{align*}
and therefore \eqref{ineq2} is completely proved.

To conclude, let us show the validity of \eqref{ineq3}. At first, let us prove that, for each $t>0$,
\begin{equation}\label{limsupq}
\limsup_{q\rightarrow 4^-}\sup_{\|u\|=t}(\|u\|_p^p+\|u\|_q^q-bt^q)\leq \sup_{\|u\|=t}(\|u\|_p^p+\|u\|_4^4)-bt^4.
\end{equation}
Indeed, let $t\in (0,+\infty)$ and let  $q\in [q_0,4)$. Since $u\in \W \mapsto \|u\|_p^p+\|u\|_q^q$ is a sequentially weakly continuous and a $C^1$ functional with no non-zero critical point, there exists $u_q\in \W$, with $\|u_q\|=t$, such that
\begin{equation*}
\|u_q\|_p^p+\|u_q\|_q^q=\sup_{\|u\|\leq t}(\|u\|_p^p+\|u\|_q^q)=\sup_{\|u\|=t}(\|u\|_p^p+\|u\|_q^q).
\end{equation*}
In particular, taking the continuity of the function $m\in [1,4]\mapsto c_m^m$  into account (see \cite[Theorem 9]{afi}), one has
\begin{align*}
t^pc_p^p &\leq \sup_{\|u\|=t}(\|u\|_p^p+\|u\|_q^q)=\|u_q\|_p^p+\|u_q\|_q^q \\
&\leq |\Omega|^{\frac{q-p}{q}}\|u_q\|_q^p + \|u_q\|_q^q\\
&\leq t^p|\Omega|^{\frac{q-p}{q}}c_q^p+t^qc_q^q \leq \sup_{m\in [1,4]}\left(t^p|\Omega|^{\frac{m-p}{m}}c_m^p + t^m c_m^m\right)<+\infty.
\end{align*}
As a result, we can find two positive constants $K_1,K_2>0$ such that
\begin{equation*}
K_1\leq \|u_q\|_q\leq K_2, \quad \text{for each } q\in [q_0,4),
\end{equation*}
and hence,
\begin{equation}\label{limitq4}
\lim_{q\rightarrow 4^-}\|u_q\|_q^{4-q}=1.
\end{equation}
Since for each $q\in[q_0,4)$ one has
\begin{align*}
\sup_{\|u\|=t}(\|u\|_p^p+\|u\|_q^q)-bt^q &=\|u_q\|_p^p+\|u_q\|_q^q-bt^q\\
&\leq \|u_q\|_p^p + |\Omega|^{\frac{4-q}{4}}\|u_q\|_4^q-bt^q\\
&=\|u_q\|_p^p+\|u_q\|_4^4 - \|u_q\|_4^q\|u_q\|_4^{4-q}+|\Omega|^{\frac{4-q}{4}}\|u_q\|_4^q -bt^q\\
&\leq \sup_{\|u\|=t}(\|u\|_p^p+\|u\|_4^4) + \left\|u_q \right\|_4^q\left(  -|\Omega|^{-\frac{(4-q)^2}{4q}} \left\|u_q \right\|_q^{4-q}
+|\Omega|^{\frac{4-q}{4}}\right)  -bt^q,
\end{align*}
then, in view of $(\ref{limitq4})$, the limit \eqref{limsupq} easily follows. Thanks to the positivity in $[q_0,4)$ of the function
\begin{equation*}
h_t(q):=\frac{a}{\sup_{\|u\|=t}\left(\|u\|_p^p+\|u\|_q^q\right)-bt^q},
\end{equation*}
the limit $(\ref{limsupq})$  implies
\begin{equation*}
h_t(4)\leq \liminf_{q\rightarrow 4^-} h_t(q),
\end{equation*}
from which
\begin{equation*}
\Lambda_{p,4}=\sup_{t>0}(h_t(4))^{\frac{4-p}{2}}\leq \liminf_{q\rightarrow 4^-} \sup_{t>0}(h_t(q))^{\frac{q-p}{2}}=\liminf_{q\rightarrow
4^-}\Lambda_{p,q}.
\end{equation*}
\end{proof}

\begin{lemma}\label{Jlambda}
One has $\lambda<\Lambda_{p,q}$ if and only if  there exists $r>0$ such that
\begin{equation}\label{condizioniMP}
\inf_{\partial B(0,r)}{J_{\lambda,q}} >0,
\end{equation}
where $\partial B(0,r):=\{u\in\W: \left\| u\right\|=r\}$.
\end{lemma}

\begin{proof}
Suppose $\lambda<\Lambda_{p,q}$. Then, there exist $\eta,t>0$ such that
$$
\lambda+\eta <\left[\frac{at^2}{\sup_{\|u\|=t}\left(\|u\|_p^p+\|u\|_q^q\right)-bt^q}\right]^{\frac{q-p}{q-2}}.
$$
For every $u\in\W$, with $\|u\|=t$, one has
\begin{equation}\label{stimaIlambda}
    \begin{split}
        J_{\lambda,q}(\lambda^{\frac{1}{q-p}} u) & = a\lambda^{\frac{2}{q-p}}\| u\|^2 + b\lambda^{\frac{q}{q-p}}\|u\|^q
        - \lambda^{\frac{q}{q-p}}\| u\|^q_q -\lambda^{1+\frac{p}{q-p}}\| u\|^p_p\\
        &=\lambda^{\frac{2}{q-p}}\left[a\|u\|^2+ \lambda^{\frac{  q-2}{q-p}}\left(b\|u\|^q- \| u\|_q^q-\|u\|_p^p\right)\right]\\
        &\geq \lambda^{\frac{2}{q-p}}at^2\left[1-\lambda^{\frac{q-2}{q-p}}\frac{\sup_{\|v\|=t}(\| v\|_q^q+\|v\|_p^p)-bt^q}{at^2}\right]\\
        &\geq\lambda^{\frac{2}{q-p}}at^2\left[1-\left(\frac{\lambda}{\lambda+\eta}\right)^{\frac{q-2}{q-p}}\right]>0.
      \end{split}
    \end{equation}
Since $\lambda^{\frac{1}{q-p}}u$ is an arbitrary function of $\partial B(0,r)$, with $r=t\lambda^{\frac{1}{q-p}}$, it follows that $\inf_{\partial B(0,r)}J_{\lambda,q} >0$.

Conversely, assume that for some $r>0$, one has $\sigma:=\inf_{\partial B(0,r)}{J_{\lambda,q}} >0$. Put $t=r\lambda^{-\frac{1}{q-p}}$. For each $n\in
\N$, with $n\geq 2$, we can choose $u_n\in \partial B(0,t)$ such that
$$
0<\left(1-\frac{1}{n}\right)\left[\sup_{\|u\|=t}\left(\|u\|_p^p+\|u\|_q^q\right)-bt^q\right]<\|u_n\|_p^p+\|u_n\|_q^q-bt^q.
$$
Then,
\begin{equation*}
    \begin{split}
        \sigma\leq J_{\lambda,q}(\lambda^{\frac{1}{q-p}} u_n) & =\lambda^{\frac{2}{q-p}}\left[at^2- \lambda^{\frac{q-2}{q-p}}\left(\| u_n\|_q^q+\|u_n\|_p^p-bt^q\right)\right]\\
        &\leq \lambda^{\frac{2}{q-p}}\left[at^2-\lambda^{\frac{q-2}{q-p}}\left(1-\frac{1}{n}\right)\left(\sup_{\|u\|=t}\left(\|
        u\|_q^q+\|u\|_p^p\right)-bt^q\right)\right].
     \end{split}
\end{equation*}
It follows that
\begin{equation*}
\begin{split}
\lambda^{\frac{q-2}{q-p}} & \leq \frac{n}{n-1}\cdot \frac{1}{\sup_{\|u\|=t}\left(\|
        u\|_q^q+\|u\|_p^p\right)bt^q}\left(at^2 -\sigma\lambda^{-\frac{2}{q-p}}\right) \\
        &\leq \frac{n}{n-1}\left(  \sup_{t>0}\frac{at^2}{\sup_{\|u\|=t}\left(\|
        u\|_q^q+\|u\|_p^p\right)-bt^q} - \frac{\sigma\lambda^{-\frac{2}{q-p}}}{\sup_{\|u\|=t}\left(\|u\|_q^q+\|u\|_p^p\right)-bt^q}\right)
 \end{split}
\end{equation*}
for each $n\in \N$, with $n\geq 2$. Passing to the limit as $n\rightarrow +\infty$, we get
$$
\lambda^\frac{q-2}{q-p} \leq \Lambda_{p,q}^\frac{q-2}{q-p} - \frac{\sigma\lambda^{-\frac{2}{q-p}}}{\sup_{\|u\|=t}\left(\|u\|_q^q+\|u\|_p^p\right)-bt^q},
$$
from which we deduce that $\lambda<\Lambda_{p,q}$.
\end{proof}

\begin{lemma}\label{nehari}
If $p\in (1,2)$ and $\lambda<\Lambda_{p,q}$, then $\mathcal{N}_{\lambda,q}^+\neq \emptyset$,  $\mathcal{N}_{\lambda,q}^-\neq \emptyset$ and
$\mathcal{N}_{\lambda,q}=\mathcal{N}_{\lambda,q}^-\cup \mathcal{N}_{\lambda,q}^+$. In particular, if  $u\in \W\setminus\{0\}$, the equation
$J_{\lambda,q}(tu)=0$ admits:
\begin{itemize}
    \item[$(i)$]  a unique positive solution $t_1=t_1(p,q,u)$ if $b\|u\|\geq\|u\|_q^q$, and one has $t_1u\in \mathcal{N}_{\lambda,q}^+$;
    \item[$(ii)$] exactly two positive solutions $t_1=t_1(p,q,u)$ and $t_2=t_2(p,q,u)$, if $b\|u\|^q<\|u\|_q^q$, and one has
    $t_1<t_2$, $t_1u\in \mathcal{N}_{\lambda,q}^+$ and $t_2u\in \mathcal{N}_{\lambda,q}^-$.
\end{itemize}
\end{lemma}

\begin{proof}
Let $u\in \W\setminus\{0\}$ be such that $b \|u\|^q<\|u\|_q^q$ (such a $u$ exists because $b<c_q^q$) and consider the $C^1$-function
$g_u:(0,+\infty)\rightarrow \R$ defined by
\begin{equation}\label{fibermap}
g_u(t)=J_{\lambda,q}(tu), \quad \text{for all } t>0.
\end{equation}
Since $1<p<2<q$ and $b \|u\|^q<\|u\|_q^q$, it is an easy matter to see that $g_u(t)<0$, for $t$ small enough and also for $t$ large enough. Moreover,
since $\lambda <\Lambda_{p,q}$, by Lemma \ref{Jlambda} we know that for some $r>0$ one has $g_u\left(\frac{r}{\|u\|}\right)>0$. In particular, $g_u$
must be increasing in some interval. Consequently, since
$$
g_u'(t)=t^{p-1}\left(2at^{2-p}\|u\|^2+qt^{q-p}(b\|u\|^q-\|u\|_q^q)-\lambda p \|u\|_p^p)\right),
$$
and the function
$$
t\in (0,\infty) \mapsto 2at^{2-p}\|u\|^2+qt^{q-p}(b\|u\|^q-\|u\|_q^q)-\lambda p \|u\|_p^p
$$
is negative near $0$ and near $+\infty$, increasing in $(0,t_0)$ and decreasing in $(t_0,+\infty)$, where
$$
t_0=\left(\frac{2(2-p)}{q(q-p)}\cdot \frac{a\|u\|^2}{\|u\|_q^q-b\|u\|^q}\right)^{\frac{1}{q-2}},
$$
we infer that there exist $\tau_u,t_u>0$, with $\tau_u<t_0<t_u$, such that $g_u'$ is negative in $(0,\tau_u)$, positive in $(\tau_u,t_u)$ and
negative in $(t_u,+\infty)$. Therefore, $g_u$ admits a unique global maximum point at $t_u$, with $g_u(t_u)>0$, and, consequently, $g_u$ admits
exactly two zeros $t_1\in (\tau_u,t_u)$ and $t_2\in (t_u,+\infty)$, such that $g'(t_1)>0$ and $g'(t_2)<0$. In particular, $t_1u\in
\mathcal{N}_{\lambda,q}^+$ and $t_2u\in\mathcal{N}_{\lambda,q}^-$.

If $b\|u\|^q\geq \|u\|_q^q$, there exists $s_u$ such that $g'_u$ is negative in $(0,s_u)$ and positive in $(s_u,+\infty)$. Hence, $g_u$ admits a unique global minimum at $s_u$, and since $g_u(t)$ is negative for $t>0$ small and $g_u(t)>0$ for $t>0$ large, we infer that $g_u$ admits a unique zero $t_1\in (s_u,+\infty)$ such that  $g_u'(t_1)>0$. In particular, $t_1u\in \mathcal{N}_{\lambda,q}^+$.

Finally, let $u\in \mathcal{N}_{\lambda,q}$. Since $g_u(1)=0$, then, by what shown above, it follows that either $t_1=1$ or $t_2=1$ and so, in any case, $u\in \mathcal{N}_{\lambda,q}^-\cup \mathcal{N}_{\lambda,q}^+$.
\end{proof}

%%%%%%%%%%%%%%%%%%%%%%%%%%%%%%%%%%%%%%%%%%
\section{Minimization on $\mathcal{N}_{\lambda,q}^+$}\label{minNpiu}

\begin{lemma}\label{nehariinf}
Let $p\in (1,2)$ and $\lambda<\Lambda_{p,q}$.  Then,
$$
\inf_{\mathcal{N}_{\lambda,q}^+}I_{\lambda,q}\leq -a\frac{(q-2)(2-p)}{2pq}\sup_{u\in \mathcal{N}_{\lambda,q}^+}\|u\|^2<0,
$$
for each $q\in [q_0,4]$.
\end{lemma}

\begin{proof}
Let $q\in [q_0,4]$ and let $u\in \mathcal{N}_{\lambda,q}^+$. Then, the following holds
\begin{align*}
& a\|u\|^2+b \|u\|^q -\|u\|_q^q-\lambda \|u\|_p^p=0, \quad \text{and} \\
& 2a\|u\|^2+q(b\|u\|^q-\|u\|_q^q)-p\lambda \|u\|_p^p >0.
\end{align*}
Equivalently,
\begin{align*}
&a\|u\|^2+b\|u\|^q-\|u\|_q^q-\lambda \|u\|_p^p=0, \quad  \text{and} \\
&a(q-2)\|u\|^2 < \lambda (q-p)\|u\|_p^p.
\end{align*}
Thus,
\begin{align*}
I_{\lambda,q}(u)=a\left(\frac{1}{2}-\frac{1}{q}\right)\|u\|^2-\lambda \left(\frac{1}{p}-\frac{1}{q}\right)\|u\|_p^p <
-a\frac{(q-2)(2-p)}{2pq}\|u\|^2,
\end{align*}
from which the conclusion follows.
\end{proof}

For each $q\in [q_0,4]$, we denote by
$$
f_q:[0,+\infty)\rightarrow [0,+\infty)
$$ 
the inverse of  the (strictly increasing) function
$$
t\in
[0+\infty) \mapsto at^2+bt^q.
$$
Moreover, we also denote by $\Psi_{\lambda,q},K_{\lambda,q},\widetilde{I}_{\lambda,q}:\W\rightarrow \R$ the functionals defined by
\begin{align}
\Psi_{\lambda,q}(u) &:= \|u\|_q^q+\lambda \|u\|_p^p, \nonumber\\
K_{\lambda,q}(u) &:= 2a(f_q(\Psi_{\lambda,q}(u))^2+qb(f_q(\Psi_{\lambda,q}(u))^q-q\|u\|_q^q-\lambda p\|u\|_p^p, \label{K}\\
\widetilde{I}_{\lambda,q}(u)&:=\frac{a}{2}(f_q(\Psi_{\lambda,q}(u))^2+\frac{b}{q}(f_q(\Psi_{\lambda,q}(u))^{q}-\frac{1}{q}\|u\|_q^q-\frac{\lambda}{p}\|u\|_p^p,
\label{Itilde}
\end{align}
for  each $u\in \W$. Notice that $f_q$ is $C^\infty$ in $(0,+\infty)$ and $(f_q)^\sigma$ is $C^1$ in $[0,+\infty)$ for each $\sigma\in [2,+\infty)$, with
$$
f_q'(y)=\frac{1}{2af_q(y)+qb(f_q(y))^{q-1}},
$$
for all $y\in (0,+\infty$), and, in particular,
\begin{equation*}
[(f_q)^2]'(y)=\frac{2}{2a+qb(f_q(y))^{q-2}}, \quad \text{and} \quad [(f_q)^q]'(y)=\frac{q(f_q(y))^{q-2}}{2a+qb(f_q(y))^{q-2}},
\end{equation*}
for all $y\in [0,+\infty)$. Since $u \mapsto \|u\|_m^m$ is $C^1$ in $\W$  for each $m\in (1,4]$, and sequentially weakly continuous in $\W$ for each $m\in [1,4)$, so are the functionals $K_{\lambda,q},\widetilde{I}_{\lambda,q}$. Moreover, by a direct calculation, one has (see page 6 of \cite{av})
\begin{equation}\label{Iprime}
\widetilde{I}_{\lambda,q}'(u)(\varphi)=\int_\Omega \left\{\left[q g_1(\Psi_{\lambda,q}(u))-1\right]u^{q-1}+\lambda
\left[pg_1(\Psi_{\lambda,q}(u))-1\right]u^{p-1}\right\}\varphi dx
\end{equation}
and
\begin{equation}\label{Kprime}
K_{\lambda,q}'(u)(\varphi)=\int_\Omega \left\{\left[q g_2\left(\Psi_{\lambda,q}(u)\right)-q^2\right]u^{q-1}+\lambda
\left[pg_2\left(\Psi_{\lambda,q}(u)\right)-p^2\right]u^{p-1}\right\}\varphi dx
\end{equation}
for each $u,\varphi\in\W$, where
\begin{equation}\label{g12}
g_1(y):=\frac{a+bf_q(y)^{q-2}}{2a+qbf_q(y)^{q-2}}, \quad \text{and} \quad g_2(y):=\frac{4a+q^2bf_q(y)^{q-2}}{2a+qbf_q(y)^{q-2}},
\end{equation}
for each $y\in [0,+\infty)$. The proof of the next lemma is along the same lines as the one of Lemma 8 of \cite{av}.

\begin{lemma}\label{KI}
For $q<4$, the functionals $\widetilde{I}_{\lambda,q},K_{\lambda,q}$ have no non-zero critical points.
\end{lemma}

Let us now introduce the following sets:
\begin{equation}\label{Alambda}
\begin{split}
\mathcal{A}_{\lambda,q}&:=\{u\in \W: J_{\lambda,q}(u)\leq 0\},\smallskip\\
\mathcal{A}_{\lambda,q}^*&:=\left\{u\in \mathcal{A}_{\lambda,q}: K_{\lambda,q}(u)\geq 0 \right\}.
\end{split}
\end{equation}

\begin{lemma}\label{astar}
Assume $q<4$, $p\in (1,2)$ and $\lambda <\Lambda_{p,q}$. The sets $\mathcal{A}_{\lambda,q}$ and $\mathcal{A}_{\lambda,q}^*$  are sequentially weakly closed and
\begin{equation}\label{coer}
\lim_{u\in \mathcal{A}_{\lambda,q}^*,\atop \|u\|\rightarrow +\infty }\widetilde{I}_{\lambda,q}(u)=+\infty.
\end{equation}
\end{lemma}

\begin{proof}
Since $J_{\lambda,q}$ is sequentially weakly lower semicontinuous  and $K_{\lambda,q}$ is sequentially weakly continuous, then $\mathcal{A}_{\lambda,q}$ and $\mathcal{A}_{\lambda,q}^*$ are sequentially weakly closed.

To finish the proof, it remains to show \eqref{coer}. Indeed, if $u\in \mathcal{A}_{\lambda,q}^*$, one has
\begin{equation}\label{ineq1}
a\|u\|^2+b\|u\|^q\leq \|u\|_q^q+\lambda \|u\|_p^p,
\end{equation}
which implies
\begin{equation}\label{ineqfq}
f_q(\Psi_{\lambda,q}(u))\geq \|u\|.
\end{equation}
In addition, one has
$$
2a(f_q(\Psi_{\lambda,q}(u))^2+qb(f_q(\Psi_{\lambda,q}(u))^q-q\|u\|_q^q-\lambda p\|u\|_p^p\geq 0,
$$
or, equivalently,
$$
\|u\|_q^q\leq \frac{2a}{q}(f_q(\Psi_{\lambda,q}(u))^2+b(f_q(\Psi_{\lambda,q}(u))^q-\lambda \frac{p}{q}\|u\|_p^p.
$$
Taking \eqref{ineqfq} and the fact that $\frac{1}{2}-\frac{2}{q^2}>0$ into account, the previous inequality implies
\begin{align*}
\widetilde{I}_{\lambda,q}(u) &\geq
a\left(\frac{1}{2}-\frac{2}{q^2}\right)(f_q(\Psi_{\lambda,q}(u))^2-\lambda\left(\frac{1}{p}-\frac{p}{q^2}\right)\|u\|_p^p\\
&\geq a\left(\frac{1}{2}-\frac{2}{q^2}\right)\|u\|^2-\lambda\left(\frac{1}{p}-\frac{p}{q^2}\right)\|u\|_p^p\\
& \geq a\left(\frac{1}{2}-\frac{2}{q^2}\right)\|u\|^2-\lambda c_p^p\left(\frac{1}{p}-\frac{p}{q^2}\right)\|u\|^p,
\end{align*}
from which the conclusion follows.
\end{proof}

\begin{proof}[Proof of Theorem \ref{subcrit}]
Since $\lambda <\Lambda_{p,q}$, by Lemma \ref{astar}, $\mathcal{A}_{\lambda,q}^*$ is weakly closed and $\widetilde{I}_{\lambda,q}$ is coercive on $\mathcal{A}_{\lambda,q}^*$. Moreover, by Lemma \ref{nehari} one has also that $\mathcal{N}_{\lambda,q}^+\neq \emptyset$. We notice also that
\begin{equation}\label{eqN}
\mathcal{N}_{\lambda,q}^+ = \mathcal{N}_{\lambda,q}\cap \mathcal{A}_{\lambda,q}^*.
\end{equation}
%\red{[Maybe ? To have also the second inclusion we need the strict inequality $>$ that we do not have, no?]}
Indeed,  $u\in \mathcal{N}_{\lambda,q}^+$ if and only if
\begin{align*}
&a\|u\|^2+b\|u\|^q = \|u\|_q^q+\lambda \|u\|_p^p = \Psi_{\lambda,q}(u), \quad \text{and}\\
&2a\|u\|^2+qb\|u\|^q-q\|u\|_q^q-\lambda p\|u\|_p^p> 0,
\end{align*}
and the above relations are, in turn, equivalent to
\begin{align*}
&f_q(\Psi_{\lambda,q}(u))=\|u\|, \quad\text{and}\\
&2a(f_q(\Psi_{\lambda,q}(u))^2+qb(f_q(\Psi_{\lambda,q}(u))^q-q\|u\|_q^q-\lambda p\|u\|_p^p>0,
\end{align*}
i.e., $u\in  \mathcal{N}_{\lambda,q}\cap \mathcal{A}_{\lambda,q}^*$.
%Now, it is easy to see that
%\red{\begin{equation*}
%I_{\lambda,q}(u)\leq \lambda \frac{q-p}{q}\left[\frac{1}{2}-\frac{1}{p}\right]\|u\|_p^p<0, \quad \text{for each }u\in \mathcal{N}_{\lambda,q}^+.
%\end{equation*}
%[già stimato nel Lemma \ref{nehariinf} con norma $\left\|u \right\|^2$, no?]}
Considering that
\begin{equation}\label{II}
\widetilde{I}_{\lambda,q}=I_{\lambda,q} \quad\text{on } \mathcal{N}_{\lambda,q}^+,
\end{equation}
and taking also Lemma \ref{nehariinf} into account, we derive the existence of $u_{\lambda,q}\in  \mathcal{A}_{\lambda,q}^*$ such that
\begin{equation}\label{II1}
\widetilde{I}_{\lambda,q}(u_{\lambda,q})=\inf_{\mathcal{A}_{\lambda,q}^*}\widetilde{I}_{\lambda,q}\leq
\inf_{\mathcal{N}_{\lambda,q}^+}\widetilde{I}_{\lambda,q}=\inf_{\mathcal{N}_{\lambda,q}^+}I_{\lambda,q}<0.
\end{equation}
In particular, one has $u_{\lambda,q}\neq 0$. Moreover, since $I_{\lambda,q}(u)=I_{\lambda,q}(|u|)$ and $u\in \mathcal{A}_{\lambda,q}^*$ implies
$|u|\in \mathcal{A}_{\lambda,q}^*$, we can also assume $u_{\lambda,q}$ non-negative.

Now, let us show that
\begin{equation}\label{un}
u_{\lambda,q}\in \mathcal{N}_{\lambda,q}.
\end{equation}
Indeed, if not, then
\begin{equation*}
a\|u_{\lambda,q}\|^2+b\|u_{\lambda,q}\|^{q}<\|u_{\lambda,q}\|_4^4+\lambda \|u_{\lambda,q}\|_p^p,
\end{equation*}
and so, in particular, $u_{\lambda,q}$ would be a local minimum of the restriction of $\widetilde{I}_{\lambda,q}$ to the set
$$
\{u\in \W\setminus \{0\}:K_{\lambda,q}(u)\geq 0\}.
$$
Now, notice that, by Lemma \ref{KI}, the set
$$\{u\in \W\setminus \{0\}: K_{\lambda,q}(u)=0\}
$$
is a $C^1$-manifold. Therefore, by the Lagrange Multipliers Theorem, we infer that there exists $\mu \in \R$ such that
%\begin{equation}\label{eqcrit} %\label{crit}
$$
\widetilde{I}_{\lambda,q}'(u_{\lambda,q})(\varphi)+\mu K_{\lambda,q}'(u_{\lambda,q})(\varphi)=0
$$
%\end{equation}
for each $\varphi \in \W$. By $\eqref{Iprime}$ and $\eqref{Kprime}$ we then get
$$
\nu_1 u_{\lambda,q}^{q-1}+\lambda \nu_2 u_{\lambda,q}^{p-1}=0 \quad \text{a.e. in }\Omega,
$$
where
\begin{align*}
\nu_1 &:=[qg_1(\Psi_{\lambda,q}(u_{\lambda,q}))-1 + \mu(qg_2(\Psi_{\lambda,q}(u_{\lambda,q}))-q^2)],\\
\nu_2
&:=[pg_1(\Psi_{\lambda,q}(u_{\lambda,q}))-1 + \mu(pg_2(\Psi_{\lambda,q}(u_{\lambda,q}))-p^2)].
\end{align*}
If at least one of the numbers $\nu_1,\nu_2$ were non-zero, the previous equality would imply that $u_{\lambda,q}\in \W$ is constant in $\Omega$, and then $u_{\lambda,q}=0$, a contradiction. Therefore $\nu_1=\nu_2=0$, from which it follows
\begin{align*}
g_1(\Psi_{\lambda,q}(u_{\lambda,q}))+\mu g_2(\Psi_{\lambda,q}(u_{\lambda,q})) &=\mu q+\frac{1}{q},\\
g_1(\Psi_{\lambda,q}(u_{\lambda,q}))+\mu g_2(\Psi_{\lambda,q}(u_{\lambda,q}))&=\mu p+\frac{1}{p},
\end{align*}
that implies $\mu= \frac{1}{pq}$ and
\begin{equation*}
g_1(\Psi_{\lambda,q}(u_{\lambda,q}))+\frac{1}{pq} g_2(\Psi_{\lambda,q}(u_{\lambda,q}))=\frac{1}{p}+\frac{1}{q}.
\end{equation*}
Using the definition of $g_1,g_2$ given in \eqref{g12},  we get
\begin{equation*}
pq \frac{a+bf_q(\Psi_{\lambda,q}(u_{\lambda,q}))^{q-2}}{2a+qbf_q(\Psi_{\lambda,q}(u_{\lambda,q}))^{q-2}}+\frac{4a+q^2bf_q(\Psi_{\lambda,q}(u_{\lambda,q}))^{q-2}}{2a+qbf_q(\Psi_{\lambda,q}(u_{\lambda,q}))^{q-2}}=p+q,
\end{equation*}
from which one easily infers that $p=2$, a contradiction. Therefore, condition $\eqref{un}$ holds and from Lemma \ref{nehari} we also have
\begin{equation}\label{un+-}
u_{\lambda,q}\in \mathcal{N}_{\lambda,q}^+\cup \mathcal{N}_{\lambda,q}^-.
\end{equation}
Moreover, condition $\eqref{un}$ also implies
$$
f_q(\Psi_{\lambda,q}(u_{\lambda,q}))=\|u_{\lambda,q}\|
$$
and, consequently, from $u_{\lambda,q}\in
\mathcal{A}_{\lambda,q}^*$ and $\eqref{un+-}$, it follows
\begin{equation*}
u_{\lambda,q}\in \mathcal{N}_{\lambda,q}^+ \ \ (\subseteq \mathcal{A}_{\lambda,q}^*).
\end{equation*}
Hence,
\begin{equation*}
\widetilde{I}_{\lambda,q}(u_{\lambda,q})=\inf_{\mathcal{A}_{\lambda,q}^*}\widetilde{I}_{\lambda,q} =\inf_{\mathcal{N}_{\lambda,q}^+}\widetilde{I}_{\lambda,q}
\end{equation*}
that, jointly to \eqref{II1}, gives \eqref{infI}.

A standard application of the Lagrange Multipliers Theorem shows that  $u_{\lambda,q}$ is a critical point of $I_{\lambda,q}$. Finally, since $u_{\lambda,q}$ is non-zero and non-negative, by the Strong Maximum Principle, we also infer that $u_{\lambda,q}$ is positive in $\Omega$. Therefore,  $u_{\lambda,q}$ is a solution of problem \eqref{problem}. The proof is now complete.
\end{proof}

\medskip

\begin{proof}[Proof of Theorem \ref{crit}]
Since $\lambda<\Lambda_{p,4}$, by $\eqref{ineq3}$ of Lemma \ref{Lambdabound}, we can find a sequence $\{q_n\}$ in $(q_0,4)$ such that $q_n\rightarrow 4^-$ and $\lambda <\Lambda_{p,q_n}$, for each $n\in \N$. To simplify the notations, we denote the functional $I_{\lambda,q_n}$ by $I_n$, and the set
$\mathcal{N}_{\lambda,q_n}^+$ by $\mathcal{N}_{n}^+$. By Theorem \ref{subcrit}, for each $n\in \N$ there exists a solution $u_n\in
\mathcal{N}_{n}^+$ of problem $(P_{\lambda,q_n})$, satisfying
\begin{equation}\label{Inun}
I_n(u_n) = a\left(\frac{1}{2}-\frac{1}{q_n}\right)\|u_n\|^2-\lambda\left(\frac{1}{p}-\frac{1}{q_n}\right)\|u_n\|_p^p=\inf_{u\in
\mathcal{N}_{n}^+}I_n(u)<0.
\end{equation}
By \eqref{Inun} it follows the boundedness of $\{u_n\}$. Thus, up to a subsequence, we can assume that there exist $l\in [0,+\infty)$ and $u_*\in\W$, such that
$$
\begin{array}{ll}
\|u_n\|  \to l, & \smallskip\\
u_n\rightharpoonup u_* \text{ in } \W, & \quad\text{ as } n\to +\infty. \smallskip\\
u_n \rightarrow u_*  \text{ in } L^m(\Omega) \ \text{ for each } m\in [1,4), &
\end{array}
$$
This implies, via Lebesgue's Dominated Convergence Theorem, that
\begin{equation}\label{criticalustar}
0=\lim_{n\rightarrow +\infty}I_n'(u_n)(\varphi)=(a+bl^2)\int_\Omega \nabla u_*\nabla \varphi dx-\int_\Omega (u_*)^3\varphi dx-\lambda\int_\Omega (u_*)^{p-1}\varphi dx,
\end{equation}
for each $\varphi \in C_0^{\infty}(\Omega)$. By density, the previous equality actually holds for each $\varphi \in \W$ and in particular for $\varphi=u_*$, so that one has
\begin{eqnarray}\label{equ*}
0=(a+bl^2)\|u_*\|^2-\|u_*\|_4^4-\lambda \|u_*\|_p^p.
\end{eqnarray}
In addition, by the Concentration-Compactness Lemma, there exist two nonnegative measures $d\mu,d\nu$, two sequences $\{\mu_k\}_{k\in \N}$, $\{\nu_k\}_{k\in \N}$ in $(0,+\infty)$, and a sequence $\{x_k\}_{k\in \N}$ in $\overline{\Omega}$ such that
\begin{equation}\label{conv}
\left\{
\begin{array}{ll}
|\nabla u_n|^2 \rightarrow d\mu \geq |\nabla u_*|^2+\displaystyle\sum_{k\in A}\mu_k\delta_{x_k} &\text{weakly}^*-\text{in the sense of measures},\\
u_n^4 \rightarrow d\nu= u_*^4+\displaystyle\sum_{k\in A}\nu_k\delta_{x_k},  & \text{weakly}^*-\text{in the sense of measures},
\end{array}\right.
\end{equation}
with $\mu_k^2 S^{-2}\geq \nu_k>0$ for each $k\in A$, where $A\subseteq\N$ is an at most countable set. It follows, in particular, that
\begin{equation}\label{lineq}
l^2\geq \|u_*\|^2+\sum_{k\in A}\mu_k.
\end{equation}
We claim that $A=\emptyset$. Indeed, assume that $A\neq\emptyset$, fix $k\in A$, %$\mu_k>0$ for some $k\in \N$
and let $\rho>0$. Testing $I_n'(u_n)(\varphi)=0$ with $\varphi=u_n\psi_\rho$, where $\psi_\rho:\R^4 \rightarrow [0,1]$ is a $C^1$-function such that
$\sup_\Omega |\nabla \psi_\rho|\leq \frac{2}{\rho}$, $\psi_\rho(x)=1$ if $|x-x_k|\leq \rho$ and $\psi_\rho(x)=0$ if $|x-x_k|\geq 2\rho$, and letting
$n\rightarrow +\infty$ and $\rho\rightarrow 0$, we get (see the proof of Theorem 3 of \cite{av})
\begin{equation*}
\mu_k\geq S^2(a+bl^2).
\end{equation*}
Plugging this inequality in $(\ref{lineq})$, it follows
\begin{eqnarray}\label{lineq1}
l^2\geq \frac{\|u_*\|^2}{1-bS^2}+\frac{aS^2}{1-bS^2}
\end{eqnarray}
Moreover, passing to the limit as $n\rightarrow +\infty$ in \eqref{Inun}, one gets
\begin{equation*}
l^2\leq \lambda \frac{4-p}{ap}\|u_*\|_p^p\leq  \lambda \frac{4-p}{ap}c_p^p\|u_*\|^p,
\end{equation*}
which, together with \eqref{lineq1}, gives
\begin{equation*}
\frac{\|u_*\|^2}{1-bS^2}+\frac{aS^2}{1-bS^2}\leq \lambda \frac{4-p}{ap}c_p^p\|u_*\|^p.
\end{equation*}
This inequality in turn implies $u_*\neq0$ and
\begin{align*}
\lambda & \geq \frac{ap}{c_p^p(4-p)(1-bS^2)}\left(\|u_*\|^{2-p}+aS^2 \|u_*\|^{-p}\right)\\
&\geq  \frac{ap}{c_p^p(4-p)(1-bS^2)} \inf_{\tau>0}\left(\tau^{2-p}+aS^2 \tau^{-p}\right)\\
 &=\frac{2a^\frac{4-p}{2}}{c_p^p(4-p)(1-bS^2)}\left(\frac{pS^2}{2-p}\right)^{\frac{2-p}{2}}
\end{align*}
against the choice of $\lambda$. Therefore it must be $A=\emptyset$ and from \eqref{conv} it follows that
\begin{equation}\label{limitun}
\lim_{n\rightarrow +\infty}\|u_n\|_4^4=\|u_*\|_4^4.
\end{equation}
Now, for every $n\in\N$ one has
\begin{align*}
0  &=I'_n(u_n)(u_n)=a\|u_n\|^2+b\|u_n\|^{q_n} - \|u_n\|_{q_n}^{q_n}-\lambda\|u_n\|_p^p\\
&\geq a\|u_n\|^2+b\|u_n\|^{q_n} -|\Omega|^{\frac{4-q_n}{4}}\|u_n\|_4^{q_n} -\lambda\|u_n\|_p^p
\end{align*}
and hence
$$
a\|u_n\|^2+b\|u_n\|^{q_n} \leq|\Omega|^{\frac{4-q_n}{4}}\|u_n\|_4^{q_n} +\lambda\|u_n\|_p^p.
$$
Passing to the limit as $n\to +\infty$ and taking account of \eqref{limitun}, we then get
\begin{equation*}
(a+bl^2)l^2\leq \|u_*\|_4^4+\lambda \|u_*\|_p^p
\end{equation*}
which, together with \eqref{equ*} and \eqref{lineq} implies
$$
\|u_*\|^2 \leq l^2\leq \|u_*\|^2,
$$
i.e., $u_n\rightarrow u_*$ strongly in $\W$. From $l=\|u_*\|$ and \eqref{criticalustar}, we conclude that $u_*$ is a non-negative critical point of $I_{\lambda,4}$.

It remains to show that $u_*\in \mathcal{N}_{\lambda,4}^+$ and
\begin{equation*}
I_{\lambda,4}(u_*)=\inf_{\mathcal{N}_{\lambda,4}^+}I_{\lambda,4}<0.
\end{equation*}
To this end, observe at first  that by the strong convergence of $\{u_n\}$ to $u_*$, it promptly follows
\begin{equation}\label{limitun1}
\lim_{n\rightarrow +\infty}I_n(u_n)=I_{\lambda,4}(u_*).
\end{equation}
Now, let $u\in \mathcal{N}_{\lambda,4}^+$, and let $t_n:=t_1(p,q_n,u)$ be as in Lemma $\ref{nehari}$. One has $t_n u\in \mathcal{N}_n^+$, that
is
\begin{equation}\label{un4}
\begin{split}
&at_n^2 \|u\|^2+t_n^{q_n}(b\|u\|^{q_n}-\|u\|_{q_n}^{q_n})-\lambda t_n^p \|u\|_p^p=0,\\
&2at_n^2\|u\|^2+ q_n t_n^{q_n}(b\|u\|^{q_n}-\|u\|_{q_n}^{q_n})-\lambda pt_n^p \|u\|_p^p>0.
\end{split}
\end{equation}
By using the above relations and the fact that $t^{q_n}<t^2$ for $t\in(0,1)$ and $\left\| u\right\|^{q_n} < 1+ \left\| u\right\|^4 $, we deduce the following two-sided estimate
$$
\min\{1,\eta_u\}\leq t_n < \left[ \frac{\lambda (q_0-p)c_p^p}{a(q_0-2)}\right]^\frac{1}{2-p}\cdot\frac{1}{\|u\|},
$$
where $\eta_u:=\left[ \frac{\lambda\left\|u\right\|_p^p }{a\left\|u \right\|^2 +b(1+\left\| u\right\|^4 ) }\right]^\frac{1}{2-p}$ is the (unique) positive solution of the equation
$$
\left( a\|u\|^2+b (1+\|u\|^4)\right)\eta^2 -\lambda \|u\|_p^p\eta^p=0.
$$
Therefore, we can assume that $t_n\rightarrow t_*>0$, and passing to the limit as $n\rightarrow +\infty$ in \eqref{un4}, we obtain
\begin{equation*}
\begin{split}
&at_*^2 \|u\|^2+t_*^4(b\|u\|^{4}-\|u\|_{4}^{4})-\lambda t_*^p \|u\|_p^p=0,\\
&2at_*^2\|u\|^2+4t_*^4(b\|u\|^{4}-\|u\|_{4}^{4})-\lambda pt_*^p \|u\|_p^p\geq 0.
\end{split}
\end{equation*}
Since $\lambda <\Lambda_{p,4}$ and $u\in \mathcal{N}_{\lambda,4}^+$, by Lemma \ref{nehari} we know that the previous inequality is strict and that
$t_*=1$. Therefore, since one has
\begin{equation*}
I_n(u_n)\leq I_n(t_n u), \quad \text{for all } n\in \N,
\end{equation*}
in view of \eqref{limitun1} we get, passing to the limit as $n\rightarrow +\infty$,
\begin{equation*}
I_{\lambda,4}(u_*) \leq I_{\lambda,4}(t^*u)=I_{\lambda,4}(u).
\end{equation*}
Consequently, being $u$ an arbitrary function in $\mathcal{N}_{\lambda,4}^+$, we have in view of Lemma \ref{nehariinf}
\begin{equation*}
I_{\lambda,4}(u_*)\leq \inf_{\mathcal{N}_{\lambda,4}^+}I_{\lambda,4}<0.
\end{equation*}
This means that $u_*\neq 0$ which, together with \eqref{equ*} and $l=\|u_*\|$, implies $u_*\in \mathcal{N}_{\lambda,4}$. Moreover, since
$$
2a\left\| u_n\right\|^2 + bq_n \left\| u_n\right\|^{q_n} - q_n\left\|u_n \right\|_{q_n}^{q_n}-\lambda p \left\| u_n\right\|_p^p >0, \quad \text{for all } n\in\N,
$$
taking the limit as $n\to +\infty$ and in view of Lemma \ref{nehari}, one has $u_*\in\mathcal{N}_{\lambda,4}^+$. Of course, being $u^*$ non-zero and
non-negative, one has that $u^*$ is positive in $\Omega$ by the Strong Maximum Principle. Therefore, $u_{\lambda,4}:=u_*$ is a solution of $(P_{\lambda,4})$. The proof is now complete.
\end{proof}

%%%%%%%%%%%%%%%%%%%%%%%%%%%%%%%%%%%%%%%
\section{Minimization on $\mathcal{N}_{\lambda,q}^-$}\label{minNmeno}
%Let us introduce the number
%\begin{equation*}
%\widehat{\Lambda}_{p,q}:=\frac{p}{2}\cdot\frac{d_{p,q}}{\widetilde{c}_{p,q}^p(c_q^q-b)^\frac{2-p}{q-2}},
%\end{equation*}
%where $d_{p,q}$ is defined in \eqref{dpq} and
%\begin{equation*}
%\widetilde{c}_{p,q}:=\sup_{u\in\mathcal{B}_q}\|u\|_p,
%\end{equation*}
%$\mathcal{B}_q$ being defined in \eqref{bq}.
In what follows, $\widehat{\Lambda}_{p,q}$ is the number defined in \eqref{lambdahatpq}.

\begin{lemma}\label{L*lim}
If the function $\psi_p:(0,S^{-2})\rightarrow \R$ defined by 
$$
\psi_p(\beta):=\sup_{\|u\|=1,\, \|u\|_4^4>\beta}\|u\|_p
$$ 
is
continuous at $b$, then
\begin{equation*}
\widehat{\Lambda}_{p,4}\leq \liminf_{q\rightarrow 4^-}  \widehat{\Lambda}_{p,q}.
\end{equation*}
\end{lemma}

\begin{proof}
We claim that
\begin{equation}\label{cpqtilde}
\widetilde{c}_{p,4}\geq \limsup_{q\rightarrow 4^-}\widetilde{c}_{p,q}.
\end{equation}
Indeed, if on the contrary we assume that
$$
\widetilde{c}_{p,4}<\limsup_{q\rightarrow 4^-}\widetilde{c}_{p,q},
$$
then there exist $\delta>0$ and two sequences $\{q_n\}\subset[q_0,4)$, $\{u_n\}\subset\W$, such that
$$
\|u_n\|=1, \quad \|u_n\|_{q_n}^{q_n}>b, \quad\text{and}\quad \widetilde{c}_{p,4}+\delta <\|u_n\|_p,
$$
for all $n\in \N$. Since $\|u_n\|_{q_n}^{q_n}>b$ implies $\|u_n\|_{4}^{4}>|\Omega|^{\frac{q_n-4}{q_n}}b^{\frac{4}{q_n}}$, by the continuity of $\psi_p$ at $b$ it follows that, for sufficiently large $n\in\N$,
$$
\|u_n\|_p\leq \sup_{\|u\|=1, \,
\|u\|_4^4>|\Omega|^{\frac{q_n-4}{q_n}}b^{\frac{4}{q_n}}}\|u\|_p<\widetilde{c}_{p,4}+\frac{\delta}{2}<\|u_n\|_p-\frac{\delta}{2},
$$
a contradiction. Therefore, \eqref{cpqtilde} holds and taking the continuity of the function $q\in [q_0,4]\mapsto c_q^q$ into account, the conclusion easily follows.
\end{proof}

\begin{lemma}\label{nehariinfmeno}
Let $p\in (1,2)$ and $\lambda<\min\{\Lambda_{p,q},\widehat{\Lambda}_{p,q}\}$.  Then,
\begin{equation}\label{infNmeno+}
\begin{split}
     I_{\lambda,q}(u)&\geq a\left(\frac{1}{2}-\frac{1}{q}\right)\left( 1-
      \lambda \widehat{\Lambda}_{p,q}^{-1}\right) \|u\|^2, \quad \text{and}\\
      \|u\|^2& > \left[\frac{a(2-p)}{(q-p)(c^q_q-b)}\right]^{\frac{2}{q-2}},%>0
    \end{split}
\end{equation}
    for each $u\in \mathcal{N}_{\lambda,q}^-$.
\end{lemma}

\begin{proof}
   Let $u\in \mathcal{N}_{\lambda,q}^-$. Then,
    \begin{equation}\label{inequ}
        \begin{split}
        & a\|u\|^2+b\|u\|^q -\|u\|_q^q-\lambda \|u\|_p^p=0, \quad \text{and}  \\
        &2a\|u\|^2+q\left( b\|u\|^q -\|u\|_q^q\right) -\lambda p\|u\|_p^p<0.
        \end{split}
    \end{equation}
    It follows that
    \begin{equation}\label{ineq4}
    a(2-p)\|u\|^2 < (q-p)(\|u\|_q^q-b\|u\|^q)\leq (q-p)(c^q_q-b)\|u\|^q.
    \end{equation}
    Therefore, $u/\|u\|\in \mathcal{B}_q$ and
    $$
    \|u\| > \left[\frac{a(2-p)}{(q-p)(c^q_q-b)}\right]^{\frac{1}{q-2}}.
    $$
    Consequently, in view \eqref{inequ}, we get
\begin{equation}\label{energyonNminus}
        \begin{split}
       I_{\lambda,q}(u) &=a\left(\frac{1}{2}-\frac{1}{q}\right)\|u\|^2-\lambda \left(\frac{1}{p}-\frac{1}{q}\right)\|u\|_p^p \\
       &\geq a\left(\frac{1}{2}-\frac{1}{q}\right)\|u\|^2-
      \lambda\widetilde{c}_{p,q}^{p}\left(\frac{1}{p}-\frac{1}{q}\right)\|u\|^p\\
      &= \left[a\left(\frac{1}{2}-\frac{1}{q}\right)-
      \lambda \widetilde{c}_{p,q}^p\left(\frac{1}{p}-\frac{1}{q}\right)\|u\|^{p-2}\right]\|u\|^2 \\
      & > \left[a\left(\frac{1}{2}-\frac{1}{q}\right)-
      \lambda \widetilde{c}_{p,q}^{p}\left(\frac{1}{p}-\frac{1}{q}\right)\left(\frac{a(2-p)}{(q-p)(c_q^q-b)}\right)^{\frac{p-2}{q-2}}\right]\|u\|^2\\
      &=a\left(\frac{1}{2}-\frac{1}{q}\right)\left( 1-
      \lambda \widehat{\Lambda}_{p,q}^{-1}\right) \|u\|^2\\
      &\geq a\left(\frac{1}{2}-\frac{1}{q}\right)\left( 1-\lambda \widehat{\Lambda}_{p,q}^{-1}\right) \left[\frac{a(2-p)}{(q-p)(c^q_q-b)}\right]^{\frac{2}{q-2}}.
      \end{split}
\end{equation}
\end{proof}

For all $R>0$, define
\begin{align*}
\widehat{\mathcal{A}}_{\lambda,q}&:=\left\lbrace u\in\mathcal{A}_{\lambda,q}\setminus\{0\}: K_{\lambda,q}(u)\leq 0\right\rbrace,\\
C_{q,R} &:=\left\lbrace u\in\W:\left\|u\right\|_q\leq R \right\rbrace,
\end{align*}
where $\mathcal{A}_{\lambda,q}$ is the set introduced in \eqref{Alambda} and $K_{\lambda,q}$  is the functional introduced in \eqref{K}.
\\[2mm]
Similar to \eqref{eqN}, we can notice that
\begin{equation}\label{nminus}
\mathcal{N}_{\lambda,q}^- = \mathcal{N}_{\lambda,q}\cap\widehat{\mathcal{A}}_{\lambda,q},
\end{equation}
and so, if $\lambda <\Lambda_{p,q}$, by Lemma \ref{nehari} one has $\widehat{\mathcal{A}}_{\lambda,q}\neq \emptyset$. We now prove a result analogous to the one of Lemma \ref{astar} for $\widehat{\mathcal{A}}_{\lambda,q}$.

\begin{lemma}\label{ahat}
Let $q\in [q_0,4)$, $q'\in [q,4)$, $R>0$, and $\lambda<\Lambda_{p,q}$. The set $\widehat{\mathcal{A}}_{\lambda,q}$ is sequentially weakly closed, and the set $\widehat{\mathcal{A}}_{\lambda,q}\cap C_{q',R}$ is sequentially weakly compact.
\end{lemma}

\begin{proof}
Let $\{u_n\}$ be a sequence in $\widehat{\mathcal{A}}_{\lambda,q}$ weakly converging to $u^*\in \W$. Let us show that $u^*\in
\widehat{\mathcal{A}}_{\lambda,q}$. For each $n\in \N$, one has
\begin{align}
        & a\|u_n\|^2+b\|u_n\|^q -\|u_n\|_q^q-\lambda \|u_n\|_p^p\leq 0, \quad \text{and}\label{un1}\\
        &2af_q(\Psi_{\lambda,q}(u_n))^2+qbf_q(\Psi_{\lambda,q}(u_n))^{q} -q\|u_n\|_q^q-\lambda p\|u_n\|_p^p\leq 0.\label{un2}
\end{align}
By using the sequential weak lower semicontinuity of $J_{\lambda,q}$ and the sequential weak continuity of $K_{\lambda,q}$, we get
%Passing to the limit as $n\rightarrow +\infty$
\begin{align*}
        & a\|u^*\|^2+b\|u^*\|^q -\|u^*\|_q^q-\lambda \|u^*\|_p^p\leq 0, \quad \text{and} \\
        &2af_q(\Psi_{\lambda,q}(u^*))^2+qbf_q(\Psi_{\lambda,q}(u^*))^q -q\|u\|_q^q-\lambda p\|u^*\|_p^p\leq 0,
    \end{align*}
and so, to prove that $u^*\in \widehat{\mathcal{A}}_{\lambda,q}$, it remains to show that $u^*\neq 0$. To this end, let $n\in \N$ and note that by the definition of $f_q$, one has
$$
af_q(\Psi_{\lambda,q}(u_n))^2+bf_q(\Psi_{\lambda,q}(u_n))^q=\Psi_{\lambda,q}(u_n)=\|u_n\|_q^q+\lambda \|u_n\|_p^p.
$$
%\blue{
%   [Miei conti.] Inserting this equality in \eqref{un2}, we get
%   $$
%   a(2-p)f_q(\Psi_{\lambda,q}(u_n))^2 + b(q-p)f_q(\Psi_{\lambda,q}(u_n))^q \leq (q-p)\|u_n\|_q^q.
%   $$
%   Observe also that by \eqref{un1} and the definition of $f_q$ one has
%   $$
%   f_q(\Psi_{\lambda,q}(u_n))\geq \|u_n\|.
%   $$
%   Consequently, one has
%   \begin{equation*}
%       a(2-p)c_q^{-2}\|u_n\|_q^2 \leq a(2-p)\|u_n\|^2 \leq (q-p)\|u_n\|_q^q
%   \end{equation*}
%   and being $u_n\neq 0$, it also follows
%   \begin{equation*}
%       \|u_n\|_q\geq \left[\frac{a(2-p)c_q^{-2}}{q-p}\right]^{\frac{1}{q-2}}.
%   \end{equation*}
%   Passing to the limit as $n\rightarrow +\infty$ we get $u^*\neq 0$.
%}
Inserting this equality in \eqref{un2}, we get
\begin{equation}\label{un3}
a(2-p)f_q(\Psi_{\lambda,q}(u_n))^2 \leq (q-2)(\|u_n\|_q^q-\|u_n\|^q)\leq (q-2)(1-bc_q^{-q})\|u_n\|_q^q.
\end{equation}
Observe also that by \eqref{un1} and the definition of $f_q$ one has
\begin{equation*}
f_q(\Psi_{\lambda,q}(u_n))\geq \|u_n\|.
\end{equation*}
Consequently, by \eqref{un3}, one has
\begin{equation*}
a(2-p)c_q^{-2}\|u_n\|_q^2 \leq a(2-p)\|u_n\|^2
\leq (q-2)(\|u_n\|_q^q-\|u_n\|^q)\leq (q-2)(1-bc_q^{-q})\|u_n\|_q^q,
\end{equation*}
and being $u_n\neq 0$, it also follows that
\begin{equation*}
\|u_n\|_q\geq \left[\frac{a(2-p)c_q^{-2}}{(q-2)(1-bc_q^{-q})}\right]^{\frac{1}{q-2}}.
\end{equation*}
Passing to the limit as $n\rightarrow +\infty$ we finally get
\begin{equation*}
\|u^*\|_q\geq \left[\frac{a(2-p)c_q^{-2}}{(q-2)(1-bc_q^{-q})}\right]^{\frac{1}{q-2}}>0,
\end{equation*}
that is $u^*\neq 0$.

Now, let $R>0$ and $q'\in [q,4)$. It is immediate to check that that $\widehat{\mathcal{A}}_{\lambda,q}\cap C_{q',R}$ is a bounded set. Moreover, the set $C_{q',R}$ is, like $\widehat{\mathcal{A}}_{\lambda,q}$, sequentially weakly closed. We conclude that $\widehat{\mathcal{A}}_{\lambda,q}\cap
C_{q',R}$ is sequentially weakly compact.
\end{proof}

\medskip

Let $\delta>0$, $x_0\in \Omega$ be such that $\overline{B}(x_0,\delta):=\{x\in\R^4:|x-x_0|\leq\delta\}\subset\Omega$ and let $\varphi\in C^\infty_0(\Omega)$ be such that
$\varphi\geq 0$, $\varphi=1$ in $\overline{B}(x_0,\delta)$. For each $\varepsilon>0$, consider the functions $v_\varepsilon,u_\varepsilon:\Omega\to\R$ defined by
\begin{equation}\label{truncinstanton}
v_\varepsilon(x):=\frac{\varphi(x)}{\varepsilon+|x-x_0|^2}, \quad \text{for all } x\in\Omega, \qquad
u_\varepsilon:=\frac{v_\varepsilon}{\left\|v_\varepsilon\right\|}.
\end{equation}
It is known that, as $\varepsilon\to 0^+$, the following estimates hold:
\begin{equation}\label{stimevepsilon}
    \left\|v_\varepsilon\right\|^2= \frac{C}{\varepsilon} + O(1), \quad S^{-2}=S^{-2}\left\| u_\varepsilon\right\|^4 \geq \left\|u_\varepsilon \right\|_4^4 =S^{-2} + O(\varepsilon),
\end{equation}
where $C>0$ is a suitable constant (see \cite{bn}).

\begin{lemma}
There exists $\varepsilon^*=\varepsilon^*(p)>0$ and two constants $C_1,C_2>0$ such that
\begin{equation}\label{normaLpuepsilon}
    C_1\varepsilon^\frac{p}{2} \leq \left\| u_\varepsilon\right\|_p^p \leq C_2 \varepsilon^\frac{p}{2}
\end{equation}
for each $\varepsilon \in (0,\varepsilon^*]$.
\end{lemma}

\begin{proof}
Since $N=4$ and $-2p>-4$, one has $v_\varepsilon^p\leq |(\cdot)-x_0|^{-2p}\in L^1(\Omega),$ for each $\varepsilon>0$, and then
$$
\lim_{\varepsilon\to 0^+}\left\| v_\varepsilon\right\|_p^p= \int_\Omega \left(\frac{\varphi(x)}{|x-x_0|^2}\right)^pdx.
$$
The conclusion then follows by the first estimate in \eqref{stimevepsilon}.
\end{proof}

Fix $\eta \in \left(0,\min\{1,S^{-2}-b\}\right)$ and for each $\varepsilon>0$ and $\lambda \in (0,\Lambda_*]$, where $\Lambda_*$ is as in \eqref{L*}, set $\varepsilon_\lambda:=\varepsilon\lambda^{\frac{2}{2-p}}$.

\begin{lemma}\label{boundh}
There exists $\bar{\varepsilon}\in (0,\varepsilon^*]$ such that
\begin{equation}\label{uel}
\frac{a}{\|v_{\varepsilon_\lambda}\|^2}+b+\eta<\|u_{\varepsilon_\lambda}\|_4^4,
\end{equation}
for each $\lambda \in (0,\Lambda_*]$ and $\varepsilon \in (0,\bar{\varepsilon}]$. Moreover, the non-negative function $h:(0,\bar{\varepsilon}]\times
(0,\Lambda_*]\rightarrow [0,+\infty)$ defined by
$$ h(\varepsilon,\lambda):=\frac{a^2}{4\varepsilon_\lambda}\left(\frac{1}{\left\|u_{\varepsilon_\lambda}
\right\|_4^4-b}-\frac{1}{S^{-2}-b}\right),
$$
is bounded in $(0,\bar{\varepsilon}]\times(0,\Lambda_*]$.
\end{lemma}
\begin{proof}
The conclusion is a straightforward consequence of \eqref{stimevepsilon}.
\end{proof}

\begin{lemma}\label{eps}
There exists $\varepsilon\in (0,\bar{\varepsilon}]$ such that
\begin{equation}\label{defgamma}
\sup_{(\rho,t)\in (0,\bar{\varepsilon}]\times(0,\Lambda_*]}h(\rho,t)\varepsilon^{\frac{2-p}{2}}
<\frac{1}{p2^p}C_1\left(\frac{a}{S^{-2}-b}\right)^\frac{p}{2},
\end{equation}
\begin{equation}\label{defgamma1}
\frac{\Lambda_*^{\frac{2}{2-p}}}{p2^p}C_1\left(\frac{a}{S^{-2}-b}\right)^\frac{p}{2}\varepsilon^{\frac{p}{2}}<\frac{a^2}{8(S^{-2}-b)},
\end{equation}
and
\begin{equation}\label{t0}
\left(\frac{2(2-p)a}{q_0(q_0-p)}\right)^{\frac{1}{q_0-2}}\cdot \frac{\|v_{\varepsilon_\lambda}\|^{-1}}{\eta^{\frac{1}{q_0-2}}}<1,
\end{equation}
for each $\lambda \in (0,\Lambda_*]$.
\end{lemma}
\begin{proof}
The conclusion is a consequence of \eqref{stimevepsilon} and Lemma \ref{boundh}.
\end{proof}

\begin{lemma}\label{stimasuperioreinf}
Let $\lambda<\Lambda_{p,q}$. There exists $q_\lambda\in[q_0,4)$ and $R_\lambda>0$ such that for all $q\in [q_\lambda,4)$ and $q'\in (q,4)$ one has
$\mathcal{N}_{\lambda,q}^-\cap C_{q',R_\lambda}\neq \emptyset$ and
\begin{equation}\label{upperestiminf}
\inf_{\mathcal{N}_{\lambda,q}^-\cap \mathcal{C}_{q',R_\lambda}} I_{\lambda,q} < \frac{a^2}{4(S^{-2}-b)} -\gamma_\lambda,
\end{equation}
where $\gamma_\lambda>0$ depends only on $\lambda,p$.
\end{lemma}

\begin{proof}
Let $\varepsilon>0$ be as in Lemma \ref{eps} and let $\lambda \in (0,\Lambda_{p,q})\subseteq (0,\Lambda_*)$. By \eqref{uel} we infer that
\begin{equation}\label{buv}
a\|v_{\varepsilon_\lambda}\|^2+b\|v_{\varepsilon_\lambda}\|^{4}<\|v_{\varepsilon_\lambda}\|_{4}^{4}.
\end{equation}
Moreover, by \eqref{defgamma} one has
\begin{equation}\label{gammal}
\begin{split}
\gamma_\lambda:&= \frac{\lambda}{p2^p}C_1 \varepsilon_\lambda
^\frac{p}{2}\left(\frac{a}{S^{-2}-b}\right)^\frac{p}{2}-\frac{a^2}{4(\left\|u_{\varepsilon_\lambda}
\right\|_4^4-b)}+\frac{a^2}{4(S^{-2}-b)} \\
&=\lambda\varepsilon_\lambda^{\frac{p}{2}}\left[\frac{C_1}{p2^p}\left(\frac{a}{S^{-2}-b}\right)^\frac{p}{2}-h(\varepsilon,\lambda)\varepsilon^{\frac{2-p}{2}}\right] >0,
\end{split}
\end{equation}
and by \eqref{defgamma1} and the non-negativity of $h$ one also has
\begin{equation*}
\gamma_\lambda\leq
\varepsilon^{\frac{p}{2}}\frac{\Lambda_*^{\frac{2}{2-p}}}{p2^p}C_1\left(\frac{a}{S^{-2}-b}\right)^\frac{p}{2}<\frac{a^2}{8(S^{-2}-b)}.
\end{equation*}
Therefore,
\begin{equation*}
\gamma_\lambda\in \left(0,\frac{a^2}{8(S^{-2}-b)}\right).
\end{equation*}
Now, note that since $\left\| v_{\varepsilon_\lambda}\right\|^q\to \left\| v_{\varepsilon_\lambda}\right\|^4$ and $\left\|
u_{\varepsilon_\lambda}\right\|^q_q\to \left\| u_{\varepsilon_\lambda}\right\|^4_4$ as $q\to 4^-$, by \eqref{buv} there exists
$q_{\varepsilon_\lambda}\in[q_0,4)$ such that
\begin{equation}\label{normv}
a\|v_{\varepsilon_\lambda}\|^2+(b+\eta)\left\|v_{\varepsilon_\lambda} \right\|^q < \left\| v_{\varepsilon_\lambda}\right\|_q^q< \left\|
v_{\varepsilon_\lambda}\right\|_q^q+\lambda \|v_{\varepsilon_\lambda}\|_p^p,
\end{equation}
for every $q\in[q_{\varepsilon_\lambda},4]$.

Let $q\in[q_{\varepsilon_\lambda},4]$ and consider the function $g_{v_{\varepsilon_\lambda}}$ defined in \eqref{fibermap}. By \eqref{normv} and Lemma \ref{nehari}, we know that there exist exactly two zeros $t_1,t_2>0$ (both depending on $q$) for $g_{v_{\varepsilon_\lambda}}$, with
$g_{v_{\varepsilon_\lambda}}$ negative in $(0,t_1)\cup(t_2,+\infty)$ and positive in $(t_1,t_2)$, and with
$$
t_1<\left(\frac{2(2-p)}{q(q-p)}\cdot
\frac{a\|v_{\varepsilon_\lambda}\|^2}{\|v_{\varepsilon_\lambda}\|_q^q-b\|v_{\varepsilon_\lambda}\|^q}\right)^{\frac{1}{q-2}}=\left(\frac{2(2-p)}{q(q-p)}\cdot
\frac{a\|v_{\varepsilon_\lambda}\|^{2-q}}{\|u_{\varepsilon_\lambda}\|_q^q-b}\right)^{\frac{1}{q-2}}<t_2.
$$
Moreover, observe that \eqref{t0} and $\eta\in (0,1)$ imply
$$
\left(\frac{2(2-p)}{q(q-p)}\cdot \frac{a\|v_{\varepsilon_\lambda}\|^{2-q}}{\|u_{\varepsilon_\lambda}\|_q^q-b}\right)^{\frac{1}{q-2}}\leq
\left(\frac{2(2-p)a}{q_0(q_0-p)}\right)^{\frac{1}{q_0-2}}\cdot \frac{\|v_{\varepsilon_\lambda}\|^{-1}}{\eta^{\frac{1}{q_0-2}}}<1.
$$
Consequently, since by \eqref{normv} one has $g_{v_{\varepsilon_\lambda}}(1)<0$, it cannot be $t_2>1$, otherwise being $g_{v_{\varepsilon_\lambda}}$
positive exactly in $(t_1,t_2)$, it should be
$$
1<t_1<\left(\frac{2(2-p)}{q(q-p)}\cdot \frac{a\|v_{\varepsilon_\lambda}\|^{2-q}}{\|u_{\varepsilon_\lambda}\|_q^q-b}\right)^{\frac{1}{q-2}}<1,
$$
a contradiction. Hence $t_2\in (0,1)$. Moreover, by Lemma \ref{nehari} we know that $t_2v_{\varepsilon_\lambda}\in \mathcal{N}_{\lambda,q}^-$, and
being $t_2\in (0,1)$ we also have
\begin{equation}\label{R0}
\|t_2v_{\varepsilon_\lambda}\|_{q'}^{q'}\leq \|v_{\varepsilon_\lambda}\|_{q'}^{q'}\leq |\Omega|+\|v_{\varepsilon_\lambda}\|_4^4=:R_\lambda %(1+|\Omega|)^2(1+\|v_{\varepsilon_\lambda}\|_4^4)=:R_\lambda,
\end{equation}
for all $q'\in [q_{\varepsilon_\lambda},4)$. %\blue{[ $\|v_{\varepsilon_\lambda}\|_{q'}^{q'}\leq |\Omega|+\|v_{\varepsilon_\lambda}\|_4^4=:R_\lambda$]?}

Now, let us consider the function $\phi:(0,+\infty)\times [q_{\varepsilon_\lambda},4]\rightarrow \R$ defined by
$$
\phi(t,q):=I_{\lambda,q}(tu_{\varepsilon_\lambda}) =\frac{a}{2}t^2 -\frac{1}{q}(\left\|u_{\varepsilon_\lambda}\right\|_q^q-b )t^q
-\frac{\lambda}{p}\left\|u_{\varepsilon_\lambda} \right\|_p^p t^p,
$$
for each  $(t,q) \in (0,+\infty)\times [q_{\varepsilon_\lambda},4]$. In light of \eqref{normv}, for $q\in [q_{\varepsilon_\lambda},4]$ one has
\begin{equation}\label{Ilambdabound}
\phi(t,q)<0 \quad\text{if} \quad t>\bar{t}:= \left( \frac{2a}{\eta}\right) ^{\frac{1}{q_{\varepsilon_\lambda}-2}}.
%\frac{2a}{\eta^{\frac{1}{q_{\varepsilon_\lambda}-2}}} \quad \blue{\left( \frac{2a}{\eta}\right) ^{\frac{1}{q_{\varepsilon_\lambda}-2}}}?
\end{equation}
Therefore,
\begin{equation}\label{Ilambdamax}
\sup_{t>0}\phi(t,q)=\max_{t\in [0,\bar{t}]}\phi(t,q).
\end{equation}
Moreover, for $t<\underline{t}:=\frac{\sqrt{a}}{2\sqrt{S^{-2}-b}}$ one has
\begin{equation}\label{Ilambdabound1}
\phi(t,4)<\frac{a^2}{8(S^{-2}-b)},
\end{equation}
and, recalling \eqref{normaLpuepsilon}, for  $t\in[\underline{t},\bar{t}]$, one also has
\begin{align*}
    \phi(t,4) &\leq  \sup_{t>0}\left[\frac{a}{2}t^2 -\frac{1}{4}(\left\|u_{\varepsilon_\lambda}\right\|_4^4-b )t^4\right]
    -\frac{\lambda}{p}C_1{\varepsilon_\lambda}^\frac{p}{2}\underline{t}^p\\
    &=\frac{a^2}{4( \left\|u_{\varepsilon_\lambda} \right\|_4^4-b)} -\frac{\lambda}{p}C_1{\varepsilon_\lambda}^\frac{p}{2}\underline{t}^p\\
    &=\frac{a^2}{4(S^{-2}-b)}-\gamma_\lambda.
\end{align*}
Hence, being $\gamma_\lambda < \frac{a^2}{8(S^{-2}-b)}$,  we obtain in view of \eqref{Ilambdabound}, \eqref{Ilambdamax} and \eqref{Ilambdabound1},
\begin{equation}\label{Ilambdabound2}
\max_{t>0} \phi(t,4)= \max_{t\in [0,\bar{t}]} \phi(t,4)< \frac{a^2}{4(S^{-2}-b)} -\gamma_\lambda.
\end{equation}
Now observe that since $\phi$ is continuous in $[0,\bar{t}]\times [q_{\varepsilon_\lambda},4]$ and $[0,\bar{t}]$ is compact, one has the continuity of the marginal function
$$
q\in [q_{\varepsilon_\lambda},4] \mapsto \max_{t\in [0,\bar{t}]}\phi(t,q).
$$
Consequently, from \eqref{Ilambdabound2}, there exists $q_{\lambda}\in [q_{\varepsilon_\lambda},4)$  such that
$$
\max_{t>0}I_{\lambda,q}(tu_{\varepsilon_\lambda})=\max_{t\in [0,\bar{t}]}\phi(t,q)< \frac{a^2}{4(S^{-2}-b)} -\gamma_{\lambda},
$$
for every $q\in [q_\lambda,4]$. Finally, if $R_\lambda$ is as in \eqref{R0} and if $q,q'\in [q_\lambda,4)$, with $q'>q$, one has
$$
\inf_{\mathcal{N}_{\lambda,q}^-\cap \mathcal{C}_{q',R_\lambda}}I_{\lambda,q}\leq I_{\lambda,q}(t_2v_{\varepsilon_\lambda})\leq
\max_{t>0}I_{\lambda,q}(tu_{\varepsilon_\lambda})<\frac{a^2}{4(S^{-2}-b)} -\gamma_\lambda.
$$
\end{proof}

\begin{lemma}\label{R}
Let $\lambda<\min\{\Lambda_{p,q},\widehat{\Lambda}_{p,q}\}$, let $q_\lambda$ be as in Lemma $\ref{stimasuperioreinf}$, let $q,q'\in [q_\lambda,4]$,
with $q'>q$, and let
$$
R=R(q,q',\lambda):=\max\left\{R_\lambda,\left(\frac{aqc_{q'}^2}{2(q-2)(S^{-2}-b)(1-\lambda \widehat{\Lambda}_{p,q}^{-1})}\right)^{\frac{1}{2}}\right\}.
$$
Then, $$
\inf_{\mathcal{N}_{\lambda,q}^{-}}I_{\lambda,q}=\inf_{\mathcal{N}_{\lambda,q}^{-}\cap \mathcal{C}_{q',R}}I_{\lambda,q}.
$$
\end{lemma}

\begin{proof}
By Lemma \ref{stimasuperioreinf} we can find a minimizing sequence $\{u_n\}\subset \mathcal{N}_{\lambda,q}^{-}$ of $I_{\lambda,q}$ on
$\mathcal{N}_{\lambda,q}^{-}$  such that
$$
I_{\lambda,q}(u_n)<\frac{a^2}{4(S^{-2}-b)}, \quad \text{for each } n\in\N.
$$
By Lemma \ref{nehariinfmeno} we know that
$$
I_{\lambda,q}(u_n) \geq a\left(\frac{1}{2}-\frac{1}{q}\right)\left[1-
      \lambda \widehat{\Lambda}_{p,q}^{-1}\right]\|u_n\|^2\geq a\left(\frac{1}{2}-\frac{1}{q}\right)\left[1-
      \lambda \widehat{\Lambda}_{p,q}^{-1}\right]c_{q'}^{-2}\|u_n\|_{q'}^{2},
$$
and therefore
$$
\|u_n\|_{q'}^2\leq \frac{qa}{2(q-2)(S^{-2}-b)(1-\lambda \widehat{\Lambda}_{p,q}^{-1})}c_{q'}^2\leq R^2,
$$
that is $u_n\in \mathcal{C}_{q',R}$, for each $n\in \N$. This proves the Lemma.
\end{proof}

\medskip

\begin{proof}[Proof of Theorem \ref{solutionN-subcrit}]
Let $q,q'\in[q_\lambda,4)$, with $q'>q$, and let $R$ be as in Lemma \ref{R}.  By Lemma \ref{ahat} $\widehat{\mathcal{A}}_{\lambda,q}\cap
C_{q',R}$ is non-empty and sequentially weakly compact. Thus, the sequentially weakly continuous functional $\widetilde{I}_{\lambda,q}$ defined in
\eqref{Itilde}, restricted to $\widehat{\mathcal{A}}_{\lambda,q}\cap C_{q',R}$, admits a global minimum $u_{\lambda,q}\in
\widehat{\mathcal{A}}_{\lambda,q}\cap C_{q',R}$, which, of course, can be assumed non-negative. By an argument similar to the one of the proof of Theorem \ref{subcrit} (or rather, by exactly the same argument as in the proof of Lemma 10 of \cite{av}), one can show that $u_{\lambda,q} \in
\mathcal{N}_{\lambda,q}$. Consequently,  $u_{\lambda,q} \in \mathcal{N}_{\lambda,q}^{-}$ and, taking in mind \eqref{nminus}, one has:
$$
\widetilde{I}_{\lambda,q}(u_{\lambda,q})=\inf_{\mathcal{N}_{\lambda,q}^{-}\cap
C_{q',R}}\widetilde{I}_{\lambda,q}(u_{\lambda,q})=\inf_{\widehat{\mathcal{A}}_{\lambda,q}\cap C_{q',R}}\widetilde{I}_{\lambda,q}(u_{\lambda,q}).
$$
Since $\widetilde{I}_{\lambda,q}=I_{\lambda,q}$ on $\mathcal{N}_{\lambda,q}$, it follows, taking Lemma \ref{nehariinfmeno} and Lemma \ref{R} into account,
$$
I_{\lambda,q}(u_{\lambda,q})=\inf_{\mathcal{N}_{\lambda,q}^{-}\cap
C_{q',R}}I_{\lambda,q}(u_{\lambda,q})=\inf_{\mathcal{N}_{\lambda,q}^-}I_{\lambda,q}>0.
$$
Finally, a standard application of the Lagrange Multipliers Theorem shows that $u_{\lambda,q}$ is a non-zero and non-negative critical point of $I_{\lambda,q}$. Thus, by the Strong Maximum Principle, $u_{\lambda,q}$ is a solution of \eqref{problem}.
\end{proof}

Let $\varepsilon>0$ be as in Lemma \ref{eps} and put \begin{equation}\label{sigma}
\sigma:=a\left(\max\left\{\frac{S^p}{d_p}\varepsilon^{\frac{p(p-2)}{4}}, \frac{p2^p}{C_1}\left(H\cdot
\varepsilon^{\frac{2-p}{2}}+1\right)\right\}\right)^{-\frac{2}{p}},
\end{equation}
where
$$
d_p:=\frac{2^{3-p}}{(4-p)c_p^p}\left( \frac{p}{2-p}\right) ^{\frac{2-p}{2}}
$$
and
$$
H:=\sup_{(\rho,t)\in
(0,\bar{\varepsilon}]\times(0,\Lambda_*]}h(\rho,t).$$

\begin{proof}[Proof of Theorem \ref{soluzcriticN-}]
Suppose at first that the function $\psi_p$ introduced in Lemma \eqref{L*lim} is continuous at $b$. By \eqref{ineq3} of Lemma
\ref{Lambdabound} and Lemma \ref{L*lim} one has
$$ \lambda<\min\{\Lambda_{p,4},\widehat{\Lambda}_{p,4}\}\leq \liminf_{q\to 4^-}\min\{\Lambda_{p,q},\widehat{\Lambda}_{p,q}\}.
$$
Therefore, there exists a sequence $\{q_n\}\subset(q_0,4)$, with $q_n\to 4^-$, such that $$\lambda<\inf_{n\in
\N}\min\{\Lambda_{p,q_n},\widehat{\Lambda}_{p,q_n}\}.$$ Setting $I_n:=I_{\lambda,q_n}$, by Lemma \ref{nehariinfmeno}, Lemma \ref{stimasuperioreinf}
and Theorem \ref{solutionN-subcrit}, there exists $v_n:=v_{\lambda,q_n}\in \mathcal{N}_{n}^-:= \mathcal{N}_{\lambda,q_n}^-$ such that,
\begin{equation}\label{infN-}
\begin{split}
  \frac{a^2}{4(S^{-2}-b)}-\gamma_\lambda & >\inf_{\mathcal{N}_{n}^-}I_n= I_n(v_n)=a\left(\frac{1}{2}-\frac{1}{q_n}\right)\|v_n\|^2-\lambda\left(\frac{1}{p}-\frac{1}{q_n}\right)\|v_n\|_p^p\\
  &>a\left(\frac{1}{2}-\frac{1}{q_n}\right)(1-\lambda \widehat{\Lambda}_{p,q_n}^{-1})\|v_n\|^2,
\end{split}
\end{equation}
for each $n\in \N$. Clearly \eqref{infN-} implies the boundedness of the sequence $\{v_n\}$. Thus there exists a subsequence, still denoted by $\{v_n\}$, $l\in(0,+\infty)$ and $v_*\in\W$, such that
$$
\begin{array}{ll}
    \|v_n\|  \to l, & \smallskip\\
    v_n\rightharpoonup v_* \text{ in } \W, & \quad\text{ as } n\to +\infty. \smallskip\\
    v_n \rightarrow v_*  \text{ in } L^m(\Omega) \ \text{ for each } m\in [1,4), &
\end{array}
$$
Notice that the positivity of $l$ follows by Lemma \ref{nehariinfmeno} which implies that
\begin{equation}\label{llower}
l^2\geq \frac{2-p}{(4-p)(S^{-2}-b)}>0.
\end{equation}
By taking the limit as $n\to+\infty$ in the relation
$$
a\left\| v_n\right\|^2 + b\left\| v_n\right\|^{q_n}=\left\| v_n\right\|_{q_n}^{q_n}   +\lambda \left\| v_n\right\|_{p}^{p},
$$
we obtain
$$
al^2+bl^4 \leq S^{-2}l^4 +\lambda \left\|v_* \right\|_p^p,
$$
and hence
\begin{equation}\label{l2one}
l^2\geq\frac{a}{S^{-2}-b} - \frac{\lambda}{l^2(S^{-2}-b)}\left\|v_* \right\|_p^p.
\end{equation}
On the other hand, passing to the limit as $n\to +\infty$ in \eqref{infN-}, we get
$$
\frac{a}{4}l^2-\lambda\frac{4-p}{4p}\|v_*\|_p^p \leq \frac{a^2}{4(S^{-2}-b)}-\gamma_\lambda,
$$
which implies that
\begin{equation}\label{l2two}
l^2 \leq \frac{a}{S^{-2}-b} -\frac{4\gamma_\lambda}{a} +\lambda\frac{4-p}{ap}\|v_*\|_p^p.
\end{equation}
Notice that conditions \eqref{l2one} and \eqref{l2two} force $v_* \not\equiv 0$ in $\Omega$.

By retracing the proof of Theorem \ref{crit}, one infers that
\begin{eqnarray}\label{eqv*}
    0=(a+bl^2)\|v_*\|^2-\|v_*\|_4^4-\lambda \|v_*\|_p^p.
\end{eqnarray}
Moreover, there exist two nonnegative measures $d\mu,d\nu$, two sequences $\{\mu_k\}_{k\in \N}$, $\{\nu_k\}_{k\in \N}$ in $(0,+\infty)$, and a
sequence $\{x_k\}_{k\in \N}$ in $\overline{\Omega}$ such that
\begin{equation}\label{conv2}
    \left\{
    \begin{array}{ll}
        |\nabla v_n|^2 \rightarrow d\mu \geq |\nabla v_*|^2+\displaystyle\sum_{k\in A}\mu_k\delta_{x_k} &\text{weakly}^*-\text{ in the sense of measures},\\
        v_n^4 \rightarrow d\nu= v_*^4+\displaystyle\sum_{k\in A}\nu_k\delta_{x_k},  & \text{weakly}^*- \text{ in the sense of measures},
    \end{array}\right.
\end{equation}
with $\mu_k^2 S^{-2}\geq \nu_k>0$ for each $k\in A$, where $A\subseteq\N$ is an at most countable set. We claim that $A=\emptyset$. Indeed, assume by
contradiction that $A\neq\emptyset$ and fix $k\in A$. Then, via the same arguments as in the proof of Theorem \ref{crit}, we arrive at
\begin{equation}\label{l2lower}
    l^2\geq \frac{\|v_*\|^2 + aS^2}{1-bS^2}.
\end{equation}
%which together \eqref{llower} yields
%\begin{eqnarray*}
%&1-\lambda\widehat{\Lambda}_{p,4}^{-1}&  \leq  \left[\frac{a}{S^{-2}-b}-\frac{4\gamma}{a}\right]\cdot\frac{1-bS^2}{\left\| v_*\right\|^2 +
%aS^2}\\[2mm]& &\leq \frac{aS^2}{\|v^*\|^2+aS^2}-\frac{4\gamma}{a}\frac{1-bS^2}{\left\| v_*\right\|^2 + aS^2}\\[2mm]
%& &\leq 1-\frac{4\gamma}{a}\frac{1-bS^2}{\left\| v_*\right\|^2 + aS^2}
%\end{eqnarray*}
%from which
%$$
%\lambda \geq  \frac{4\gamma}{a}\frac{1-bS^2}{\left\| v_*\right\|^2 + aS^2}\widehat{\Lambda}_{p,4}.$$
Now, from \eqref{l2two} and \eqref{l2lower}, it follows
\begin{equation*}
\frac{\|v^*\|^2}{1-bS^2}\leq -\frac{4\gamma_\lambda}{a}+\lambda \frac{4-p}{ap}\|v^*\|_p^p\leq  -\frac{4\gamma_\lambda}{a}+\lambda
c_p^p\frac{4-p}{ap}\|v^*\|^p,
\end{equation*}
that implies
\begin{equation*}
\inf_{t>0}\left(\frac{t^2}{1-bS^2}-\lambda c_p^p \frac{4-p}{ap}t^p+\frac{4\gamma_\lambda}{a}\right)\leq 0.
\end{equation*}
Since the above infimum is attained at
$$
\bar{t}:=\left(\lambda\cdot \frac{(4-p)(1-bS^2)c_p^p}{2a}\right)^{\frac{1}{2-p}},
$$
it turns out
\begin{equation*}
\overline{t}^p\left(\frac{\bar{t}^{2-p}}{1-bS^2}-\lambda c_p^p \frac{4-p}{ap}\right)+\frac{4\gamma_\lambda}{a}\leq 0,
\end{equation*}
and plugging {\tiny }the value of $\bar{t}$ in the above inequality we get
\begin{equation*}
 [\lambda
c_p^p(4-p)]^{\frac{2}{2-p}}\left(\frac{1-bS^2}{2a}\right)^{\frac{p}{2-p}}\left(\frac{1}{p}-\frac{1}{2}\right)\geq 4\gamma_\lambda.
\end{equation*}
Therefore, recalling the definitions of $\sigma$ and $\gamma_\lambda$, given in \eqref{sigma} and \eqref{gammal}, respectively, and the fact that $b>S^{-2}-\sigma$, it follows
\begin{equation*}
\begin{split}
\lambda&\geq \frac{a^{\frac{p}{2}}d_p(\gamma_\lambda )^{\frac{2-p}{2}}}{(1-bS^2)^{\frac{p}{2}}}\\
&\geq \frac{a^{\frac{p}{2}}d_p}{S^p(S^{-2}-b)^{\frac{p}{2}}}\cdot \lambda\varepsilon^
{\frac{p(2-p)}{4}}\left[\frac{C_1}{p2^p}\frac{a^{\frac{p}{2}}}{(S^{-2}-b)^{\frac{p}{2}}}-\sup_{(\rho,t)\in (0,\bar{\varepsilon}]\times
(0,\Lambda_*]}h(\varepsilon,t)\varepsilon^{\frac{2-p}{2}} \right]^{\frac{2-p}{2}}\\
&>\lambda \frac{a^{\frac{p}{2}}d_p}{S^p}\sigma^{-\frac{p}{2}}\varepsilon^
{\frac{p(2-p)}{4}}\left[\frac{C_1a^{\frac{p}{2}}}{p2^p}\sigma^{-\frac{p}{2}}-H \varepsilon^{\frac{2-p}{2}} \right]^{\frac{2-p}{2}}\geq \lambda,
\end{split}
\end{equation*}
a contradiction. Thus, as in the proof of Theorem \ref{crit}, we infer that $v_n$ strongly converges to $v_*$ and that $v_{\lambda,4}:=v_*$ is a solution of $(P_{\lambda,4})$.

At this point, let $u\in \mathcal{N}_{\lambda,4}^-$. By \eqref{ineq4} we know that $b\|u\|^4<\|u\|_4^4$, and then, since $q_n\rightarrow 4$, we can assume  $b\|u\|^{q_n}<\|u\|_{q_n}^{q_n}$ for each $n\in \N$. Therefore, by Lemma \ref{nehari}, we can find
$$
t_n=t(q_n,u)>\left(\frac{{2}(2-p)}{q_n(q_n-p)}\cdot \frac{a\|u\|^2}{\|u\|_{q_n}^{q_n}-b\|u\|^{q_n}}\right)^{\frac{1}{q_n-2}}
$$
such that $t_nu\in \mathcal{N}_{\lambda,q_n}^-$. In particular, since
$$
a \|u\|^2 \tau^2 +(b\|u\|^{q_n}-\|u\|_{q_n}^{q_n})\tau^{q_n}-\lambda
\|u\|_p^p\tau^p<0, \quad \text{if } \tau> \left(\frac{a\|u\|^2}{\|u\|_{q_n}^{q_n}-b\|u\|^{q_n}}\right)^{\frac{1}{q_n-2}},
$$
one also has
$$
t_n\leq
\left(\frac{a\|u\|^2}{\|u\|_{q_n}^{q_n}-b\|u\|^{q_n}}\right)^{\frac{1}{q_n-2}}.
$$
Hence, the sequence $\{t_n\}$ turns out bounded and away from $0$,
and so we can assume that
$$
t_n\rightarrow t_*\in (0,+\infty).
$$
Then, passing  to the limit as $n\rightarrow +\infty$ in
\begin{align*}
& at_n^2 \|u\|^2-t_n^{q_n}(\|u\|^{q_n}_{q_n}-b\|u\|^{q_n})-\lambda t_n^p \|u\|_p^p=0,\\
&2at_n^2\|u\|^2-q_n t_n^{q_n}(\|u\|^{q_n}_{q_n}-b\|u\|^{q_n})-\lambda pt_n^p \|u\|_p^p < 0,
\end{align*}
we get
\begin{align*}
& a t_*^2 \|u\|^2-t_*^{4}(\|u\|^{4}_{4}-b\|u\|^{4})-\lambda t_*^p \|u\|_p^{p}=0\\
& 2at_*^2 \|u\|^2-4t_*^{4}(\|u\|^{4}_{4}-b\|u\|^{4})-\lambda p t_*^p \|u\|_p^p\leq 0.
\end{align*}
Since $\lambda<\Lambda_{p,4}$, by Lemma \ref{nehari} we know that the last inequality is strict, so that $t^*u\in \mathcal{N}_{\lambda,4}^-$. By the
proof of Lemma \ref{nehari}, we also know that there is a unique positive number  $t$ such that $tu\in \mathcal{N}_{\lambda,4}^-$.  Since $u\in
\mathcal{N}_{\lambda,4}^-$, we conclude that $t_*=1$.

After that, one has
$$
I_{\lambda,4}(u)=\liminf_{n\rightarrow +\infty}I_n(t_nu)\geq \liminf_{n\rightarrow +\infty}I_n(v_n)=I_{\lambda,4}(v_*).
$$
Being $u$ an arbitrary function of $\mathcal{N}_{\lambda,4}^-$, we infer that (recall Lemma \ref{nehariinfmeno}),
$$
\inf_{\mathcal{N}_{\lambda,4}^-}I_{\lambda,4}=I_{\lambda,4}(v_*)\geq \frac{a}{4}(1-\lambda\widehat{\Lambda}_{p,4}^{-1})\|v_*\|^2>0,
$$
which concludes the proof in the case $\psi_p$ is continuous at $b$.

Finally, let $b$ be any number in $(S^{-2}-\sigma,S^{-2})$. Since $\psi_p$  is non-increasing, it has at most a countable set of discontinuity points and, therefore, the set of points of $(S^{-2}-\sigma, S^{-2})$ at which $\psi_p$ is continuous is dense in the same set. Thus, we can fix a non-increasing sequence $\{b_n\}\subset (S^{-2}-\sigma,S^{-2})$ such that $b_n\rightarrow b$ and $\psi_p$ is continuous at each
$b_n$. Since $\Lambda_{p,4}$ and $\widehat{\Lambda}_{p,4}$ are non-decreasing at $b$ and $b_n\geq b$ for each $n\in \N$, we obtain
$$
\lambda<\min\{\Lambda_{p,4},\widehat{\Lambda}_{p,4}\}\leq \min\{\Lambda_{p,4,n},\widehat{\Lambda}_{p,4,n}\},
$$
where $\Lambda_{p,4,n}$ and $\widehat{\Lambda}_{p,4,n}$ are the numbers obtained replacing $b$ by $b_n$ in
$\Lambda_{p,4}$ and $\widehat{\Lambda}_{p,4}$, respectively. This means that for each $n\in\N$ we can find a critical point $v_n$ of the functional
$$
\widehat{I}_n(u):=\frac{1}{2}\|u\|^2+\frac{b_n}{4}\|u\|^4-\frac{1}{4}\|u\|_4^4-\frac{\lambda}{p} \|u\|_p^p
$$
that minimizes $\widehat{I}_n$ on 
$$
\widehat{\mathcal{N}}_n^-:=\{u\in \W\setminus\{0\}: \widehat{J}_n(u):=\widehat{I}_n'(u)(u)=0, \
\widehat{J}_n'(u)(u)<0\}.
$$ 
Arguing exactly as above, it follows that $v_n$ strongly converges to a critical point of $I_\lambda$ that minimizes
$I_{\lambda,4}$ on $\mathcal{N}_{\lambda,4}^-$.
\end{proof}

%%%%%%%%%%%%%%%%%%%%%%%%%%%%%%%%%%%
\section*{Acknowledgements}
The authors are members of the Gruppo Nazionale per l'Analisi Matematica, la Probabilità e le loro Applicazioni (GNAMPA) of
the Istituto Nazionale di Alta Matematica (INdAM). This paper was written within the INdAM-GNAMPA project CUP code E53C25002010001. 
%%%%%%%%%%%%%%%%%%%%%%%%%%%%%%%%%%%%%%%%%

\end{document}